\documentclass[3p]{elsarticle}
\usepackage{color}
\usepackage{array}
\usepackage{float}
\usepackage{multirow}
\usepackage{amsthm}
\usepackage{amsmath}
\usepackage{amssymb}
\usepackage[unicode=true,
 bookmarks=true,bookmarksnumbered=false,bookmarksopen=false,
 breaklinks=true,pdfborder={0 0 1},backref=false,colorlinks=true]
 {hyperref}
\hypersetup{
 allcolors=BlueGreen, psdextra}

\makeatletter

\providecommand{\tabularnewline}{\\}

 \theoremstyle{definition}
 \newtheorem*{defn*}{\protect\definitionname}
\theoremstyle{plain}
\newtheorem{thm}{\protect\theoremname}
  \theoremstyle{plain}
  \newtheorem{lem}[thm]{\protect\lemmaname}
  \theoremstyle{definition}
  \newtheorem{defn}[thm]{\protect\definitionname}
  \theoremstyle{plain}
  \newtheorem{prop}[thm]{\protect\propositionname}
  \theoremstyle{plain}
  \newtheorem{cor}[thm]{\protect\corollaryname}
  \theoremstyle{remark}
  \newtheorem*{rem*}{\protect\remarkname}

\usepackage{lineno}

\usepackage{diagbox}
\usepackage{graphicx}

\usepackage{listings}
\makeatother

  \providecommand{\corollaryname}{Corollary}
  \providecommand{\definitionname}{Definition}
  \providecommand{\lemmaname}{Lemma}
  \providecommand{\propositionname}{Proposition}
  \providecommand{\remarkname}{Remark}
\providecommand{\theoremname}{Theorem}

\begin{document}

\begin{frontmatter}{}

\title{Linear complexity of Ding-Helleseth generalized cyclotomic sequences
of order eight}

\author[liang]{Yana Liang}

\ead{ynliang72@163.com}

\author[cao]{Jiali Cao}

\ead{jiali\_cao@hotmail.com}

\author[chen]{Xingfa Chen}

\ead{chenxingfa@gdei.edu.cn}

\author[sysu]{Shiping Cai}

\ead{csp1519@163.com}

\author[sysu]{Xiang Fan\corref{cor1}}

\ead{fanx8@mail.sysu.edu.cn}

\cortext[cor1]{Corresponding author.}

\address[liang]{School of Mathematics and Statistics, Zhaoqing University, Zhaoqing
526061, China}

\address[cao]{School of Applied Mathematics, Guangdong University of Technology,
Guangzhou 510520, China}

\address[chen]{Department of Mathematics, Guangdong University of Education, Guangzhou
510303, China}

\address[sysu]{School of Mathematics, Sun Yat-sen University, Guangzhou 510275,
China}
\begin{abstract}
During the last two decades, many kinds of periodic sequences with
good pseudo-random properties have been constructed from classical
and generalized cyclotomic classes, and used as keystreams for stream
ciphers and secure communications. Among them are a family DH-GCS$_{d}$
of generalized cyclotomic sequences on the basis of Ding and Helleseth's
generalized cyclotomy, of length $pq$ and order $d=\mathrm{gcd}(p-1,q-1)$
for distinct odd primes $p$ and $q$. The linear complexity (or linear
span), as a valuable measure of unpredictability, is precisely determined
for DH-GCS$_{8}$ in this paper. Our approach is based on Edemskiy
and Antonova's computation method with the help of explicit expressions
of Gaussian classical cyclotomic numbers of order $8$. Our result
for $d=8$ is compatible with Yan's low bound $(pq-1)/2$ of the linear
complexity for any order $d$, which means high enough to resist security
attacks of the Berlekamp--Massey algorithm. Finally, we include SageMath
codes to illustrate the validity of our result by examples.\end{abstract}
\begin{keyword}
Generalized cyclotomic sequence\sep Linear complexity\sep Cyclotomic
number

\MSC[2010]11B50\sep 94A55\sep 94A60
\end{keyword}

\end{frontmatter}{}


\section{\label{sec:Intr}Introduction}

Pseudo-random sequences are widely used in many fields such as simulation,
software testing, channel coding, ranging systems, radar systems,
global positioning systems, spread-spectrum communication systems,
code-division multiple-access systems, and especially in stream ciphers
and secure communications. When employed for cryptography, they should
possess certain unpredictable properties, especially with a high linear
complexity. The definition of \emph{linear complexity} (or \emph{linear
span}) is the length $L$ of the shortest linear feedback shift register
(LFSR) that generates the sequence. Since the Berlekamp--Massey algorithm
\cite{Berlekamp1968algebraic,Massey1969shift} can deduce the whole
sequence from a knowledge of just $2L$ consecutive digits, a ``high''
linear complexity $L$ should be at least one half of the length (or
minimum period) of the sequence.

Certain cyclotomic sequences, such as Legendre sequences and Hall's
sextic residue sequences, possess good pseudo-random properties (especially
with high linear complexities \cite{DingHellesethShan1998linear,KimSong2001linear}),
and have been widely used as keystreams in private-key cryptosystems.
Nowadays, a well-known technique to design a sequence with a high
linear complexity is based on classical or generalized cyclotomy \cite{CusickDingRenvall2004stream}.
There are many kinds of generalized cyclotomic sequences, binary and
non-binary, of different lengths. One family of them are based on
Whiteman's generalized cyclotomy \cite{Whiteman1962family}, of length
$pq$ for distinct odd primes $p$ and $q$, with a high linear complexity
precisely determined when $\mathrm{gcd}(p-1,q-1)=2^{k}$ \cite{Ding1997linear,BaiFuXiao2005linear,YanDuXiaoHuang2009linear}.
However, the imbalance (with $q-p-1$ more $1$'s than $0$'s in each
period) makes them improper for applications.

Another family of generalized cyclotomic sequences of length $pq$
and order $d=\mathrm{gcd}(p-1,q-1)$, denoted by DH-GCS$_{d}$ in
this paper, are based on Ding and Helleseth's generalized cyclotomy
\cite{DingHelleseth1998new}. (Also see \cite{FanGe2014} for
a unified approach to Whiteman\textquoteright s and Ding-Helleseth\textquoteright s
generalized cyclotomy.) The linear complexity of DH-GCS$_{d}$ when
$d=2$, $4$, $6$ is explicitly determined in \cite{BaiLiuXiao2005linear,YanHongXiao2008linear,EdemskiyAntonova2014linear},
and ensured by Yan \cite{Yan2011linear} to be at least $(pq-1)/2$
for any order $d$. We contribute to this line by calculating the
linear complexity of DH-GCS$_{8}$, using the computation method of
Edemskiy and Antonova \cite{EdemskiyAntonova2011computation,EdemskiyAntonova2014linear}
with the help of Gaussian classical cyclotomic numbers.

First let us introduce the definition of the Ding-Helleseth generalized
cyclotomic sequences DH-GCS$_{d}$.
\begin{itemize}
\item Let $p$ and $q$ be distinct odd primes with $d=\mathrm{gcd}(p-1,q-1)$.
\item Let $e=\mathrm{lcm}(p-1,q-1)=(p-1)(q-1)/d$. 
\item By the Chinese remainder theorem, take a common primitive root $g$
of both $p$ and $q$, and take an integer $f$ such that $f\equiv g$
$(\mathrm{mod}\ p)$ and $f\equiv1$ $(\mathrm{mod}\ q)$.
\item Let $\mathbb{Z}_{pq}^{*}$ denote the multiplicative group consisting
of all invertible elements in the residue ring
\[
\mathbb{Z}_{pq}=\mathbb{Z}/(pq\mathbb{Z})=\{[m\ \mathrm{mod}\ pq]:m\in\mathbb{Z}\},
\]
where $[m\ \mathrm{mod}\ pq]$ denotes the residue class of $m$ modulo
$pq$. Then 
\[
\mathbb{Z}_{pq}^{*}=\{[m\ \mathrm{mod}\ pq]:m\in\mathbb{Z},\ \mathrm{gcd}(m,pq)=1\}.
\]

\end{itemize}
Whiteman \cite{Whiteman1962family} proved that every element of $\mathbb{Z}_{pq}^{*}$
can be written uniquely of the form $[g^{u}f^{v}$ $\mathrm{mod}\ pq]$
with $u\in\{0,1,\dots,e-1\}$ and $v\in\{0,1,\dots,d-1\}$. Here the
uniqueness means that 
\[
g^{u}f^{v}\equiv g^{u'}f^{v'}\ (\mathrm{mod}\ pq)\Rightarrow u\equiv u'(\mathrm{mod}\ e)\text{ and }v\equiv v'(\mathrm{mod}\ d).
\]

\begin{defn*}
\cite{DingHelleseth1998new}\emph{ Ding-Helleseth generalized cyclotomic
classes} of order $d$ with respect to $(p,q,g)$ are
\[
D_{i}=\{[g^{i+dt}f^{v}\ \mathrm{mod}\ pq]:t=0,1,\dots,\frac{e}{d}-1;\ v=0,1,\dots,d-1\}\subseteq\mathbb{Z}_{pq}^{*},
\]
where $i\in\mathbb{Z}$, and let
\begin{align*}
D_{i}^{(p)} & =\{[g^{i+dt}\ \mathrm{mod}\ p]:t=0,1,\dots,\frac{p-1}{d}-1\}\subseteq\mathbb{Z}_{p}^{*},\\
Q_{i}=qD_{i}^{(p)} & =\{[qg^{i+dt}\ \mathrm{mod}\ pq]:t=0,1,\dots,\frac{p-1}{d}-1\}\subseteq\mathbb{Z}_{pq}^{*},\\
D_{i}^{(q)} & =\{[g^{i+dt}\ \mathrm{mod}\ q]:t=0,1,\dots,\frac{q-1}{d}-1\}\subseteq\mathbb{Z}_{q}^{*},\\
P_{i}=pD_{i}^{(q)} & =\{[pg^{i+dt}\ \mathrm{mod}\ pq]:t=0,1,\dots,\frac{q-1}{d}-1\}\subseteq\mathbb{Z}_{pq}^{*},
\end{align*}
\end{defn*}
\begin{itemize}
\item If $i\equiv j$ $(\mathrm{mod}\ d)$, then $D_{i}=D_{j}$, $P_{i}=P_{j}$
and $Q_{i}=Q_{j}$. These definitions concern only $[i\ \mathrm{mod}\ d]$.
\item If $i\not\equiv j$ $(\mathrm{mod}\ d)$, then $D_{i}\cap D_{j}=P_{i}\cap P_{j}=Q_{i}\cap Q_{j}=\O$
(the empty set).
\item $|D_{i}|=e$, $|P_{i}|=|D_{i}^{(q)}|=(q-1)/d$, and $|Q_{i}|=|D_{i}^{(p)}|=(p-1)/d$.
\item $\mathbb{Z}_{pq}^{*}=\bigcup_{i=0}^{d-1}D_{i}$, and $\mathbb{Z}_{pq}=\{[0\ \mathrm{mod}\ pq]\}\cup\bigcup_{i=0}^{d-1}(D_{i}\cup P_{i}\cup Q_{i})$.
\end{itemize}
The \emph{Ding-Helleseth generalized cyclotomic sequence} of order
$d$ (DH-GCS$_{d}$) with respect to $(p,q,g)$ is defined (cf. \cite{BaiLiuXiao2005linear,YanHongXiao2008linear,EdemskiyAntonova2014linear})
as the sequence $s^{\infty}=(s_{i})_{i\geqslant0}=(s_{0},s_{1},s_{2},\dots,s_{i},\dots)$
with
\[
s_{i}=\begin{cases}
1, & \text{if }[i\ \mathrm{mod}\ pq]\in\bigcup_{j=d/2}^{d-1}(D_{j}\cup P_{j}\cup Q_{j}),\\
0, & \text{otherwise}.
\end{cases}
\]
This sequence possesses the minimum period $pq$, and the almost balance
of 
\[
\frac{d}{2}(e+\frac{q-1}{d}+\frac{p-1}{d})=\frac{1}{2}(pq-1)
\]
symbols $1$'s and $\frac{1}{2}(pq+1)$ symbols $0$'s in a period.

Our main result, an explicit expression of the linear complexity of
DH-GCS$_{8}$ in terms of $p$ and $q$, can be stated in the following
Theorem \ref{thm:main}, and will be proved in Section \ref{sec:LCDH8}.
\begin{thm}
\label{thm:main} Given primes $p,q$ with $\mathrm{gcd}(p-1,q-1)=8$,
regardless of the choice of a common primitive root $g$ of theirs,
the linear complexity of the Ding-Helleseth generalized cyclotomic
sequence with respect $(p,q,g)$ is
\begin{align*}
L(p,q) & =pq-1-(p-1)\cdot\begin{cases}
\frac{1}{2}, & \text{if }\left(\frac{2}{p}\right)_{8}=1,\\
0, & \text{otherwise},
\end{cases}-(q-1)\cdot\begin{cases}
\frac{1}{2}, & \text{if }\left(\frac{2}{q}\right)_{8}=1,\\
0, & \text{otherwise},
\end{cases}\\
 & \quad-(p-1)(q-1)\cdot\begin{cases}
\frac{1}{2}, & \text{if }\left(\frac{2}{p}\right)_{8}=1=\left(\frac{2}{q}\right)_{4}=1,\\
\frac{1}{2}, & \text{if }\left(\frac{2}{p}\right)_{8}=1\neq\left(\frac{2}{q}\right)_{4}\text{ and }\left(\frac{p}{q}\right)_{4}=1,\\
\frac{1}{2}, & \text{if }\left(\frac{2}{p}\right)_{4}=1\neq\left(\frac{2}{p}\right)_{8},\left(\frac{2}{q}\right)_{4}\neq1\text{ and }\left(\frac{p}{q}\right)_{2}=1\neq\left(\frac{p}{q}\right)_{4},\\
\frac{1}{4}, & \text{if }\left(\frac{2}{p}\right)_{4}=1\neq\left(\frac{2}{q}\right)_{4}\text{ and }\left(\frac{p}{q}\right)_{2}\neq1,\\
0, & \text{otherwise}.
\end{cases}
\end{align*}
\end{thm}
\begin{itemize}
\item Here $\left(\frac{c}{p}\right)_{k}$ denotes the \emph{$k$-th power
symbol}. Namely, $\left(\frac{c}{p}\right)_{k}=1\Leftrightarrow$
there exists an integer $u$ such that $u^{k}\equiv c\not\equiv0$
$(\mathrm{mod}\ p)$.
\end{itemize}

\section{Computation methods}

For a sequence $s^{\infty}=(s_{i})_{i\geqslant0}=(s_{0},s_{1},s_{2},\dots,s_{i},\dots)$
over a field $F$, its \emph{linear complexity} (or \emph{linear span})
$L(s^{\infty})$ is defined as the least positive integer $L$ such
that there exists constants $c_{1},c_{2},\dots,c_{L}\in F$ satisfying
\[
s_{i}+c_{1}s_{i-1}+c_{2}s_{i-2}+\cdots+c_{L}s_{i-L}=0,\quad\text{ for any }i\geqslant L.
\]
If such $L$ do not exist, let $L(s^{\infty})=\infty$. For a ultimately
periodic sequence, such $L$ exists. A corresponding polynomial 
\[
m(x)=1+c_{1}x+\cdots+c_{L}x^{L}
\]
 with $L=L(s^{\infty})<\infty$ is called the minimal polynomial 
of the sequence $s^{\infty}$.

For a periodic sequence $s^{\infty}=(s_{i})_{i\geqslant0}$ of period
$N$, let $S(x)=s_{0}+s_{1}x+s_{2}x^{2}+\cdots s_{N-1}x^{N-1}$. Then
the minimal polynomial and linear complexity of $s^{\infty}$ can
be given by \cite{LidlNiederreiter1997finite}:
\begin{align*}
m(x) & =(x^{N}-1)/\mathrm{gcd}(x^{N}-1,S(x)),\\
L(s^{\infty})=\deg(m(x)) & =N-\deg(\mathrm{gcd}(x^{N}-1,S(x))).
\end{align*}
Moreover, assume that $s^{\infty}$ is binary (i.e., $F=\mathrm{GF}(2)$)
and $N$ is odd. Let $\alpha$ be a primitive $N$-th root of unity
in an extension field of $\mathrm{GF}(2)$. Then $X^{N}-1=\prod_{k=0}^{N-1}(X-\alpha^{k})$,
and
\begin{align*}
m(x) & =(x^{N}-1)/\prod_{0\leqslant k\leqslant N-1:S(\alpha^{k})=0}(X-\alpha^{k}),\\
L(s^{\infty}) & =N-|\{k:S(\alpha^{k})=0,\ 0\leqslant k\leqslant N-1\}|.
\end{align*}

For the Ding-Helleseth generalized cyclotomic sequence DH-GCS$_{d}$
with minimum period $N=pq$,
\[
S(x)=\sum_{j=d/2}^{d-1}(\sum_{u\in D_{j}}x^{u}+\sum_{u\in P_{j}}x^{u}+\sum_{u\in Q_{j}}x^{u}).
\]
Let us recall a method of Edemskiy and Antonova \cite{EdemskiyAntonova2011computation,EdemskiyAntonova2014linear}
to compute the linear complexity of DH-GCS$_{d}$.
\begin{itemize}
\item Let $\alpha$ be a primitive $pq$-th root of unity in an extension
field of $\mathrm{GF}(2)$.
\item Let $\beta=\alpha^{q}$ and $\gamma=\alpha^{p}$. Then $\beta$ (resp.
$\gamma$) is a primitive $p$-th (resp. $q$-th) root of unity.\end{itemize}
\begin{defn*}
Let us introduce the cyclotomic polynomials
\[
S_{d}(x)=\sum_{u\in D_{0}^{(p)}}x^{u}=\sum_{t=0}^{\frac{p-1}{d}-1}x^{g^{dt}},\qquad T_{d}(x)=\sum_{u\in D_{0}^{(q)}}x^{u}=\sum_{t=0}^{\frac{q-1}{d}-1}x^{g^{dt}},
\]
and quote from \cite{Edemskii2010linear,EdemskiyAntonova2014linear}
the following properties for $i,j\in\{0,1,\dots,d-1\}$ and $k\in\mathbb{Z}$:\end{defn*}
\begin{itemize}
\item $\sum_{u\in D_{i}^{(p)}}\beta^{u}=S_{d}(\beta^{g^{i}})=S_{d}(\beta^{g^{i+dk}})$,
and $\sum_{u\in D_{i}^{(q)}}\gamma^{u}=T_{d}(\gamma^{g^{i}})=T_{d}(\gamma^{g^{i+dk}})$.
\item $\sum_{i=0}^{d-1}S_{d}(\beta^{g^{i}})=\sum_{i=0}^{d-1}T_{d}(\gamma^{g^{i}})=1$.
\item $\sum_{u\in Q_{i}}\alpha^{ku}=S_{d}(\beta^{kg^{i}})$, and $\sum_{u\in P_{i}}\alpha^{ku}=T_{d}(\gamma^{kg^{i}})$.
\item Let $D_{i,j}=\{[g^{j+dt}f^{i-j}\ \mathrm{mod}\ pq]:t=0,1,\dots,\frac{e}{d}-1\}\subseteq D_{j}$.
Then $\bigcup_{i=0}^{d-1}D_{i,j}=D_{j}$, $[D_{i,j}\ \mathrm{mod}\ p]\subseteq D_{i}^{(p)}$,
$[D_{i,j}\ \mathrm{mod}\ q]\subseteq D_{j}^{(q)}$, and 
\[
\sum_{u\in D_{i,j}}\alpha^{ku}=S_{d}(\beta^{kg^{i-\mathrm{ind}_{g}^{(p)}q}})T_{d}(\gamma^{kg^{j-\mathrm{ind}_{g}^{(q)}p}}),
\]
where $\mathrm{ind}_{g}^{(q)}p$ denotes the \emph{discrete logarithm}
of $p$ in the field $\mathrm{GF}(q)$ relative to the basis $g$,
i.e., $p\equiv g^{\mathrm{ind}_{g}^{(q)}p}$ $(\mathrm{mod}\ q)$,
 and similarly $q\equiv g^{\mathrm{ind}_{g}^{(p)}q}(\mathrm{mod}\ p)$.\end{itemize}
\begin{lem}
[{\cite[Theorem 1]{EdemskiyAntonova2014linear}}]\label{lem:comp}Let
$k\in\mathbb{Z}$ and $\delta_{p}(k)=\begin{cases}
1, & \text{if }k\not\equiv0\ (\mathrm{mod}\ p),\\
0 & \text{if }k\equiv0\ (\mathrm{mod}\ p).
\end{cases}$ Then 
\[
S(\alpha^{k})=\sum_{t=d/2}^{d-1}(T_{d}(\gamma^{kg^{t-\mathrm{ind}_{g}^{(q)}p}})\delta_{p}(k)+T_{d}(\gamma^{kg^{t}})+S_{d}(\beta^{kg^{t}})).
\]
\end{lem}
\begin{defn*}
For $s^{\infty}=$ DH-GCS$_{d}$, define a matrix $\mathbb{S}=(s_{i,j})_{i,j=0}^{d}$
of order $d+1$ as follows.\end{defn*}
\begin{itemize}
\item Let $s_{i,j}=S(\alpha^{k})$ if $[k\ \mathrm{mod}\ pq]\in D_{i,j}$
with $0\leqslant i,j\leqslant d-1$. As $[k\ \mathrm{mod}\ p]\in D_{i}^{(p)}$
and $[k\ \mathrm{mod}\ q]\in D_{j}^{(q)}$, $s_{i,j}=\sum_{t=d/2}^{d-1}(T_{d}(\gamma^{g^{j+t-\mathrm{ind}_{g}^{(q)}p}})\delta+T_{d}(\gamma^{g^{j+t}})+S_{d}(\beta^{g^{i+t}}))$.
\item Let $s_{i,d}=S(\alpha^{k})$ if $[k\ \mathrm{mod}\ pq]\in Q_{i-\mathrm{ind}_{g}^{(p)}q}$
with $0\leqslant i\leqslant d-1$. As $[k\ \mathrm{mod}\ p]\in D_{i}^{(p)}$
and $k\equiv0$ $(\mathrm{mod}\ q)$, $s_{i,d}=\sum_{t=d/2}^{d-1}(T_{d}(1)+T_{d}(1)+S_{d}(\beta^{g^{i+t}}))=\sum_{t=\frac{d}{2}}^{d-1}S_{d}(\beta^{g^{i+t}})$.
\item Let $s_{d,j}=S(\alpha^{k})$ if $[k\ \mathrm{mod}\ pq]\in P_{j-\mathrm{ind}_{g}^{(q)}p}$
with $0\leqslant j\leqslant d-1$. As $[k\ \mathrm{mod}\ q]\in D_{j}^{(q)}$
and $k\equiv0$ $(\mathrm{mod}\ p)$, $S_{d}(1)=|D_{0}^{(p)}|=\frac{p-1}{d}$
and $s_{d,j}=\sum_{t=d/2}^{d-1}T_{d}(\gamma^{g^{j+t}})+\frac{p-1}{2}$.
\item Let $s_{d,d}=S(\alpha^{0})=S(1)=\sum_{t=d/2}^{d-1}(T_{d}(1)+S_{d}(1))=\frac{p-1}{2}+\frac{q-1}{2}=\begin{cases}
0 & \text{if }p\equiv q\ (\mathrm{mod}\ 4),\\
1 & \text{otherwise}.
\end{cases}$\end{itemize}
\begin{defn}
For a vector $(a_{1},a_{2},\dots,a_{r})$, define a \emph{transformation}
$\sigma_{u}$ by
\[
\sigma_{u}(a_{1},a_{2},\dots,a_{r})=(a_{u+1},a_{u+2},\dots,a_{r},a_{1},a_{2},\dots,a_{u})
\]
 for an integer $0\leqslant u<r$. For any $u'\in\mathbb{Z}$ with
$u'\equiv u$ $(\mathrm{mod}\ r)$, define $\sigma_{u'}=\sigma_{u}$.
Let  
\begin{align*}
\overrightarrow{S_{d}}(x) & =(S_{d}(x),S_{d}(x^{g}),S(x^{g^{2}}),\dots,S(x^{g^{d-1}})),\\
\overrightarrow{T_{d}}(x) & =(T_{d}(x),T_{d}(x^{g}),T(x^{g^{2}}),\dots,T(x^{g^{d-1}})),\\
\overrightarrow{A_{d}}(x) & =\sum_{t=d/2}^{d-1}\overrightarrow{S_{d}}(x^{g^{t}}),\quad\overrightarrow{B_{d}}(x)=\sum_{t=d/2}^{d-1}\overrightarrow{T_{d}}(x^{g^{t}}).
\end{align*}
\end{defn}
\begin{prop}
If $4\mid d=\mathrm{gcd}(p-1,q-1)$, then we have \textup{$s_{d,d}=0$,
\[
(s_{i,d})_{i=0}^{d-1}=\overrightarrow{A_{d}}(\beta)=\sum_{t=d/2}^{d-1}\sigma_{t}(\overrightarrow{S_{d}}(\beta)),\quad(s_{d,j})_{j=0}^{d-1}=\overrightarrow{B_{d}}(\gamma)=\sum_{t=d/2}^{d-1}\sigma_{t}(\overrightarrow{T_{d}}(\gamma)),
\]
}and $s_{i,j}=s_{i,d}+s_{d,j}+s_{d,j'}$ for $0\leqslant i,j,j'\leqslant d-1$
with $j'\equiv j-\mathrm{ind}_{g}^{(q)}p$ $(\mathrm{mod}\ d)$. Note
that
\[
(s_{d,j}+s_{d,j'})_{j=0}^{d-1}=\overrightarrow{B_{d}}(\gamma)+\sigma_{-\mathrm{ind}_{g}^{(q)}p}(\overrightarrow{B_{d}}(\gamma)).
\]
\end{prop}
\begin{lem}
\label{lem:method}To calculate the linear complexity and minimal
polynomial of $s^{\infty}=$ \textup{DH-GCS}$_{d}$, it suffices to
determine the zero elements of the matrix $\mathbb{S}$. As $|D_{i,j}|=e/d$,
$|Q_{i}|=(p-1)/d$, $|P_{j}|=(q-1)/d$, we have 
\begin{align*}
L(s^{\infty}) & =pq-\frac{e}{d}\cdot|\{0\leqslant i,j\leqslant d-1:s_{i,j}=0\}|-\frac{p-1}{d}\cdot|\{0\leqslant i\leqslant d-1:s_{i,d}=0\}|\\
 & \quad-\frac{q-1}{d}\cdot|\{0\leqslant j\leqslant d-1:s_{d,j}=0\}|-\begin{cases}
1 & \text{if }p\equiv q\ (\mathrm{mod}\ 4),\\
0 & \text{otherwise}.
\end{cases}
\end{align*}
Let $\mathbf{d}(x)=\begin{cases}
x-1 & \text{if }p\equiv q\ (\mathrm{mod}\ 4),\\
1 & \text{otherwise},
\end{cases}$ $\mathbf{D}_{i,j}(x)={\displaystyle \prod_{k\in D_{i,j}}(x-\alpha^{k})}$,
$\mathbf{P}_{j}(x)={\displaystyle \prod_{k\in P_{j}}}(x-\alpha^{k})$,
and $\mathbf{Q}_{i}(x)={\displaystyle \prod_{k\in Q_{i}}}(x-\alpha^{k})$,
then the minimal polynomial of \textup{DH-GCS}$_{d}$ is
\[
m(x)=(x^{pq}-1)/(\mathbf{d}(x)\prod_{\substack{0\leqslant i,j\leqslant d-1,\\
s_{i,j}=0
}
}\mathbf{D}_{i,j}(x)\prod_{\substack{0\leqslant j\leqslant d-1,\\
s_{d,j}=0
}
}\mathbf{P}_{j-\mathrm{ind}_{g}^{(q)}p}(x)\prod_{\substack{0\leqslant i\leqslant d-1,\\
s_{i,d}=0
}
}\mathbf{Q}_{i-\mathrm{ind}_{g}^{(p)}q}(x)).
\]

\end{lem}
The choice of a primitive $pq$-th root $\alpha$ of unity gives some
flexibility for our computation.
\begin{lem}
\label{lem:choice}For any $h,l\in\{0,1,2,\dots,d-1\}$, take any
$k\in D_{h,l}$, and let $\alpha'=\alpha^{k}$, $\beta'=(\alpha')^{q}=\beta^{g^{h}}$,
$\gamma'=(\alpha')^{p}=\gamma^{g^{l}}$. Then $\alpha'$, $\beta'$,
$\gamma'$ are primitive $pq$-th, $p$-th, $q$-th root of unity
respectively. Note that $S_{d}((\beta')^{g^{i}})=S_{d}(\beta^{g^{i+h}})$,
and $T_{d}((\gamma')^{g^{i}})=T_{d}(\gamma^{g^{i+l}})$ for $0\leqslant i\leqslant d-1$.
Therefore,
\begin{alignat*}{2}
\overrightarrow{S_{d}}(\beta') & =\sigma_{h}(\overrightarrow{S_{d}}(\beta)),\qquad & \overrightarrow{A_{d}}(\beta') & =\sigma_{h}(\overrightarrow{A_{d}}(\beta)),\\
\overrightarrow{T_{d}}(\gamma') & =\sigma_{l}(\overrightarrow{T_{d}}(\gamma)), & \overrightarrow{B_{d}}(\gamma') & =\sigma_{l}(\overrightarrow{B_{d}}(\gamma)).
\end{alignat*}
\end{lem}
\begin{itemize}
\item In the following computation, $\overrightarrow{S_{d}}(\beta)$, $\overrightarrow{A_{d}}(\beta)$,
$\overrightarrow{T_{d}}(\gamma)$ and $\overrightarrow{B_{d}}(\gamma)$
can be written freely up to transformations of the form $\sigma_{h}$,
by some good choice of $\alpha$.
\end{itemize}

For an even divisor $d$ of $p-1$, the values of $S_{d}(\beta^{g^{i}})=\sum_{t=0}^{\frac{p-1}{d}-1}\beta^{g^{i+dt}}$
(with $0\leqslant i\leqslant d-1$) can be deduced from the values
of $S_{d/2}(\beta^{g^{j}})=\sum_{t=0}^{\frac{p-1}{d/2}-1}\beta^{g^{j+\frac{d}{2}t}}$
(with $0\leqslant j\leqslant d/2-1$) by the following lemma read
off from \cite[\S 3]{Edemskii2010linear}, with the help of classical
cyclotomic numbers of order $d$, introduced by Gauss in his famous
book \textquotedblleft Disquisitiones Arithmeticae\textquotedblright{}
\cite{GaussDisquisitiones}, defined as
\begin{align*}
(i,j)_{d} & =|(D_{i}^{(p)}+1)\cap D_{j}^{(p)}|\\
 & =|\{(u,v)\in\mathbb{Z}^{2}:g^{i+du}+1\equiv g^{j+dv}\ (\mathrm{mod}\ p),\ 0\leqslant u,v<\frac{p-1}{d}\}|.
\end{align*}
For fixed $(p,g)$, note that $(i,j)_{d}$ depends only on the residue
classes $[i\ \mathrm{mod}\ d]$ and $[j\ \mathrm{mod}\ d]$.
\begin{lem}
\label{lem:Sd}For $0\leqslant i\leqslant d/2-1$, we have $S_{d}(\beta^{g^{i}})+S_{d}(\beta^{g^{i+d/2}})=S_{d/2}(\beta^{g^{i}})$
and 
\begin{align*}
S_{d}(\beta^{g^{i}})S_{d}(\beta^{g^{i+d/2}}) & =\sum_{j=0}^{d/2-1}(d/2,j-i)_{d}S_{d/2}(\beta^{g^{j}})+\frac{p-1}{d}\cdot\frac{1-(-1)^{\frac{p-1}{d}}}{2}.
\end{align*}
\end{lem}
\begin{proof}
By definition, $S_{d}(\beta^{g^{i}})+S_{d}(\beta^{g^{i+d/2}})=\sum_{t=0}^{\frac{p-1}{d}-1}(\beta^{g^{i+dt}}+\beta^{g^{i+d/2+dt}})=\sum_{t=0}^{\frac{p-1}{d/2}-1}\beta^{g^{i+\frac{d}{2}t}}=S_{d/2}(\beta^{g^{i}})$,
and $S_{d}(\beta^{g^{i}})S_{d}(\beta^{g^{i+d/2}})=\sum_{u,v=0}^{\frac{p-1}{d}-1}\beta^{g^{i+du}+g^{i+d/2+dv}}$.
For fixed integers $0\leqslant j<d$ and $0\leqslant t<\frac{p-1}{d}$,
as $g^{i+du}+g^{i+d/2+dv}\equiv g^{j+dt}$ $(\mathrm{mod}\ p)$ $\Leftrightarrow$
$1+g^{d/2+d(v-u)}\equiv g^{j-i+d(t-u)}$ $(\mathrm{mod}\ p)$, we
see that
\[
(d/2,j-i)_{d}=|\{(u,v)\in\mathbb{Z}^{2}:g^{i+du}+g^{i+d/2+dv}\equiv g^{j+dt}\ (\mathrm{mod}\ p),0\leqslant u,v<\frac{p-1}{d}\}|.
\]
As $g^{i+du}+g^{i+d/2+dv}\equiv0$ $(\mathrm{mod}\ p)$ $\Leftrightarrow$
$g^{d/2+d(v-u)}\equiv-1\equiv g^{\frac{p-1}{2}}$ $(\mathrm{mod}\ p)$
$\Leftrightarrow$ $d/2+d(v-u)\equiv\frac{p-1}{2}$ $(\mathrm{mod}\ p-1)$
$\Leftrightarrow$ $v-u\equiv\frac{1}{2}(\frac{p-1}{d}-1)$ $(\mathrm{mod}\ \frac{p-1}{d})$,
we see that
\begin{align*}
\frac{p-1}{d}\cdot\frac{1-(-1)^{\frac{p-1}{d}}}{2} & =\begin{cases}
\frac{p-1}{d}, & \text{if }\frac{p-1}{d}\text{ is odd},\\
0, & \text{otherwise},
\end{cases}\\
 & =|\{(u,v)\in\mathbb{Z}^{2}:g^{i+du}+g^{i+d/2+dv}\equiv0\ (\mathrm{mod}\ p),0\leqslant u,v<\frac{p-1}{d}\}|.
\end{align*}
Also note that $1+g^{d/2+du}\equiv g^{k+dv}\ (\mathrm{mod}\ p)\Leftrightarrow g^{-d/2-du}+1\equiv g^{k-d/2+d(v-u)}\ (\mathrm{mod}\ p)$,
so $(d/2,k)_{d}=(-d/2,k-d/2)_{d}=(d/2,k+d/2)_{d}$, and in particular
$(d/2,j-i)_{d}=(d/2,j+d/2-i)_{d}$. In all, we have 
\begin{align*}
S_{d}(\beta^{g^{i}})S_{d}(\beta^{g^{i+d/2}}) & =\sum_{j=0}^{d-1}\sum_{t=0}^{\frac{p-1}{d}-1}(d/2,j-i)_{d}\beta^{g^{j+dt}}+\frac{p-1}{d}\cdot\frac{1-(-1)^{\frac{p-1}{d}}}{2}\\
 & =\sum_{j=0}^{d-1}(d/2,j-i)_{d}S_{d}(\beta^{g^{j}})+\frac{p-1}{d}\cdot\frac{1-(-1)^{\frac{p-1}{d}}}{2}\\
 & =\sum_{j=0}^{d/2-1}(d/2,j-i)_{d}S_{d/2}(\beta^{g^{j}})+\frac{p-1}{d}\cdot\frac{1-(-1)^{\frac{p-1}{d}}}{2}.\qedhere
\end{align*}
\end{proof}
\begin{cor}
\label{cor:Ad}Suppose that $\overrightarrow{A_{d}}(\beta)=(c_{0},c_{1},c_{2}\dots,c_{d-1})$.
Then $c_{0}=\sum_{t=d/2}^{d-1}S_{d}(\beta^{g^{t}})$, and 
\begin{alignat*}{2}
c_{i+1}-c_{i} & =S_{d/2}(\beta^{g^{i}}), &  & \text{for }0\leqslant i\leqslant d-2,\\
c_{j} & =\sum_{t=d/2}^{d-1}S_{d}(\beta^{g^{t}})+\sum_{i=0}^{j-1}S_{d/2}(\beta^{g^{i}}), & \quad & \text{for }1\leqslant j\leqslant d-1.
\end{alignat*}
\end{cor}
\begin{proof}
By definition, $c_{i}=\sum_{t=d/2}^{d-1}S_{d}(\beta^{g^{i+t}})$ for
$0\leqslant i\leqslant d-1$. So $c_{i+1}-c_{i}=S_{d}(\beta^{i+d})-S_{d}(\beta^{i+d/2})=S_{d}(\beta^{i})+S_{d}(\beta^{i+d/2})=S_{d/2}(\beta^{g^{i}})$
for $0\leqslant i\leqslant d-2$. Then the last expression for $c_{j}$
is clear.
\end{proof}

\section{\label{sec:A8}Calculation of \texorpdfstring{$\protect\overrightarrow{A_{8}}(\beta)$}{{A\_8$(\beta)$}}}

In this section we calculate $\overrightarrow{A_{8}}(\beta)$ (up
to transformations $\sigma_{h}$),  by Lemma \ref{lem:Sd} and Corollary
\ref{cor:Ad}, from the values of $S_{4}(\beta^{g^{j}})$ and the
Gaussian classical cyclotomic numbers $(4,j)_{8}$ for $0\leqslant j\leqslant3$.
\begin{itemize}
\item Let $p=x^{2}+4y^{2}=a^{2}+2b^{2}\equiv1$ $(\mathrm{mod}\ 8)$ be
a prime for integers $x$, $y$, $a$, $b$ with $x\equiv a\equiv1$
$(\mathrm{mod}\ 4)$.
\item As $x^{2}\equiv a^{2}\equiv1\equiv p$ $(\mathrm{mod}\ 8)$, we have
$y\equiv b\equiv0$ $(\mathrm{mod}\ 2)$.
\item As $p\equiv x^{2}$ $(\mathrm{mod}\ 16)$, we have $p\equiv1$ $(\mathrm{mod}\ 16)$
$\Leftrightarrow x\equiv1$ $(\mathrm{mod}\ 8)$; $p\equiv9$ $(\mathrm{mod}\ 16)$
$\Leftrightarrow x\equiv5$ $(\mathrm{mod}\ 8)$.
\end{itemize}
Gauss proved in 1828 that (see his Werke Vol II \cite{GaussWerkeII}):
\begin{itemize}
\item $\left(\frac{2}{p}\right)_{4}=1$ $\Leftrightarrow$ $y\equiv0$ $(\mathrm{mod}\ 4)$
$\Leftrightarrow$ $a\equiv1$ $(\mathrm{mod}\ 8)$; $\left(\frac{2}{p}\right)_{4}\neq1$
$\Leftrightarrow$ $y\equiv2$ $(\mathrm{mod}\ 4)$ $\Leftrightarrow$
$a\equiv5$ $(\mathrm{mod}\ 8)$.
\end{itemize}
As $p\equiv1$ $(\mathrm{mod}\ 8)$, $2$ is a quadratic residue modulo
$p$. Let $2\equiv g^{2u}$ $(\mathrm{mod}\ p)$ with $u\in\mathbb{Z}$.
As 
\[
S_{2}(\beta^{g^{i}})^{2}=\sum_{t=0}^{\frac{p-1}{2}-1}\beta^{2g^{i+2t}}=\sum_{t=0}^{\frac{p-1}{2}-1}\beta^{g^{i+2(t+u)}}=S_{2}(\beta^{g^{i}}),
\]
$S_{2}(\beta^{g^{i}})\in\{0,1\}$ for $i\in\{0,1\}$. Moreover, $S_{2}(\beta)+S_{2}(\beta^{g})=1$.
So $(S_{2}(\beta),S_{2}(\beta^{g}))=(1,0)$ or $(0,1)$. As
\[
16(2,0)_{4}=16(2,2)_{4}=p-3+2x,\qquad16(2,1)_{4}=16(2,3)_{4}=p+1-2x,
\]
(see \cite[\S 11]{Dickson1935cyclotomy}), by Lemma \ref{lem:Sd},
we can solve $\overrightarrow{S_{4}}(\beta)$ as follows.
\begin{lem}
[{\cite[\S 3.1]{Edemskii2010linear}}]\label{lem:S4}Up to transformations
$\sigma_{h}$ (by a good choice of $\alpha$ as in Lemma \ref{lem:choice}),
we have 
\[
\overrightarrow{S_{4}}(\beta)=\begin{cases}
(1,0,0,0) & \text{if }p\equiv1\ (\mathrm{mod}\ 16)\text{ and }\left(\frac{2}{p}\right)_{4}=1,\\
(0,1,1,1) & \text{if }p\equiv9\ (\mathrm{mod}\ 16)\text{ and }\left(\frac{2}{p}\right)_{4}=1,\\
(\mu,1,\mu+1,1) & \text{if }p\equiv1\ (\mathrm{mod}\ 16)\text{ and }\left(\frac{2}{p}\right)_{4}\neq1,\\
(\mu,0,\mu+1,0) & \text{if }p\equiv9\ (\mathrm{mod}\ 16)\text{ and }\left(\frac{2}{p}\right)_{4}\neq1.
\end{cases}
\]
\end{lem}
\begin{itemize}
\item Here $\mu$ is a root of the equation $\mu^{2}+\mu+1=0$ in $\mathrm{GF}(2^{2})=\{0,1,\mu,\mu+1\}$.
\end{itemize}
From \cite[Appendix]{Lehmer1955number} we read off Gaussian classical
cyclotomic numbers $(4,j)_{8}$ for $0\leqslant j\leqslant3$ as in
Table \ref{tab:CyclNum}. This table uses a particular choice for
the sign of $y=\pm\sqrt{\frac{p-x^{2}}{4}}$ (determined by the choice
of the primitive root $g$ modulo $p$), which makes no difference
for our result of $\overrightarrow{A_{8}}(\beta)$.
\begin{table}[h]
\protect\caption{\label{tab:CyclNum}$64(4,j)_{8}$ for $0\leqslant j\leqslant3$}

\centering{}%
\begin{tabular}{c|c|c|c|c}
\hline 
\multirow{2}{*}{} & \multicolumn{2}{c|}{$\left(\frac{2}{p}\right)_{4}=1$} & \multicolumn{2}{c}{$\left(\frac{2}{p}\right)_{4}\neq1$}\tabularnewline
\cline{2-5} 
 & $p\equiv1\ (\mathrm{mod}\ 16)$ & $p\equiv9\ (\mathrm{mod}\ 16)$ & $p\equiv1\ (\mathrm{mod}\ 16)$ & $p\equiv9\ (\mathrm{mod}\ 16)$\tabularnewline
\hline 
\hline 
$64(4,0)_{8}$ & $p-7-2x+8a$ & $p-15-2x$ & $p-7-10x$ & $p-15-10x-8a$\tabularnewline
\hline 
$64(4,1)_{8}$ & $p+1+2x-4a$ & $p-7+2x+4a$ & $p+1+2x-4a+16y$ & $p-7+2x+4a+16y$\tabularnewline
\hline 
$64(4,2)_{8}$ & $p+1-2x$ & $p-7-2x-8a$ & $p+1+6x+8a$ & $p-7+6x$\tabularnewline
\hline 
$64(4,3)_{8}$ & $p+1+2x-4a$ & $p-7+2x+4a$ & $p+1+2x-4a-16y$ & $p-7+2x+4a-16y$\tabularnewline
\hline 
\end{tabular}
\end{table}

From values of $S_{4}(\beta^{g^{j}})$ and $(4,j)_{8}$ for $0\leqslant j\leqslant3$
as above, Lemma \ref{lem:Sd} gives the values of $S_{8}(\beta^{i})+S_{8}(\beta^{i+4})$
and $S_{8}(\beta^{i})S_{8}(\beta^{i+4})$ for $0\leqslant i\leqslant3$
in $\mathrm{GF}(4)=\{0,1,\mu,\mu+1\}$. So the sets $\{S_{8}(\beta^{i}),S_{8}(\beta^{i+4})\}$
lie in Table \ref{tab:XY}, a list the sets $\{X,Y\}$ with values
of $X+Y$ and $XY$ given in $\mathrm{GF}(4)$.

\begin{table}[h]
\protect\caption{\label{tab:XY}$\{X,Y\}$ with $X+Y$ and $XY$ given in $\mathrm{GF}(4)$}

\centering{}%
\begin{tabular}{c|c|c|c|c}
\hline 
\diagbox{$X+Y$}{$XY$} & $0$ & $1$ & $\mu$ & $\mu+1$\tabularnewline
\hline 
$0$ & $\{0\}$ & $\{1\}$ & $\{\mu+1\}$ & $\{\mu\}$\tabularnewline
\hline 
$1$ & $\{0,1\}$ & $\{\mu,\mu+1\}$ & $\{\eta,\eta+1\}$ & $\{\eta^{2},\eta^{2}+1\}$\tabularnewline
\hline 
\multirow{2}{*}{$\mu$} & \multirow{2}{*}{$\{0,\mu\}$} & \multirow{2}{*}{$\{\eta^{3}+\eta,\eta^{3}+\eta^{2}\}$} & $\{\eta^{3}+\eta+1,$ & \multirow{2}{*}{$\{1,\mu+1\}$}\tabularnewline
 &  &  & $\eta^{3}+\eta^{2}+1\}$ & \tabularnewline
\hline 
$\mu+1$ & $\{0,\mu+1\}$ & $\{\eta^{3},\eta^{3}+\eta^{2}+\eta+1\}$ & $\{1,\mu\}$ & $\{\eta^{3}+1,\eta^{3}+\eta^{2}+\eta\}$\tabularnewline
\hline 
\end{tabular}
\end{table}

\begin{itemize}
\item Here $\eta$ is a root of the equation $\eta^{2}+\eta=\mu$ in $\mathrm{GF}(2^{4})$.
Note that $\eta$ is also a root of the equation $\eta^{4}+\eta+1=0$,
and $\mathrm{GF}(2^{4})=\{\sum_{i=0}^{3}\lambda_{i}\eta^{i}:\lambda_{i}=0,1\}$.
\end{itemize}
Now we calculate $\overrightarrow{A_{8}}(\beta)$ in four cases according
to $[p$ $\mathrm{mod}\ 16]$ and $\left(\frac{2}{p}\right)_{4}$.

\subsection{Case: $p\equiv1$ $(\mathrm{mod}\ 16)$ and $\left(\frac{2}{p}\right)_{4}=1$}
\begin{itemize}
\item $p\equiv1$ $(\mathrm{mod}\ 16)$ $\Leftrightarrow$ $x\equiv1$ $(\mathrm{mod}\ 8)$. 
\item $\left(\frac{2}{p}\right)_{4}=1$ $\Leftrightarrow$ $y\equiv0$ $(\mathrm{mod}\ 4)$
$\Leftrightarrow$ $a\equiv1$ $(\mathrm{mod}\ 8)$.
\item $a^{2}\equiv1$ $(\mathrm{mod}\ 16)$ $\Rightarrow2b^{2}\equiv0$
$(\mathrm{mod}\ 16)$ $\Rightarrow b\equiv0$ $(\mathrm{mod}\ 4)$.
\end{itemize}
Let $x=8x_{1}+1$, $y=4y_{1}$, $a=8a_{1}+1$, and $b=4b_{1}$, with
integers $x_{1}$, $y_{1}$, $a_{1}$ and $b_{1}$. Note that $p=(8x_{1}+1)^{2}+4(4y_{1})^{2}=(8a_{1}+1)^{2}+2(4b_{1})^{2}$
$\Rightarrow$ $16(x_{1}-a_{1})\equiv32b_{1}^{2}$ $(\mathrm{mod}\ 64)$
$\Rightarrow$ $(x_{1}-a_{1})\equiv2b_{1}^{2}$ $(\mathrm{mod}\ 4)$.
So $x_{1}\equiv a_{1}$ $(\mathrm{mod}\ 2)$. Therefore,
\[
(4,0)_{8}=\frac{1}{64}(p-7-2x+8a)=x_{1}^{2}+y_{1}^{2}+a_{1}\equiv y_{1}\ (\mathrm{mod}\ 2).
\]

By a good choice of $\alpha$, we may assume
\[
\overrightarrow{S_{4}}(\beta)=(S_{4}(\beta),S_{4}(\beta^{g}),S_{4}(\beta^{g^{2}}),S_{4}(\beta^{g^{3}}))=(1,0,0,0).
\]

\begin{itemize}
\item Note that $S_{8}(\beta)+S_{8}(\beta^{g^{4}})=1$, and $S_{8}(\beta)S_{8}(\beta^{g^{4}})=(4,0)_{8}=y_{1}$.
So
\[
\{S_{8}(\beta),S_{8}(\beta^{g^{4}})\}=\{y_{1}\mu,y_{1}\mu+1\}.
\]

\item For $i\in\{1,3\}$, $S_{8}(\beta^{g^{i}})+S_{8}(\beta^{g^{i+4}})=0$,
and $S_{8}(\beta^{g^{i}})S_{8}(\beta^{g^{i+4}})=(4,1)_{8}=(4,3)_{8}$.
So
\[
S_{8}(\beta^{g})=S_{8}(\beta^{g^{3}})=S_{8}(\beta^{g^{5}})=S_{8}(\beta^{g^{7}}).
\]

\item Note that $S_{8}(\beta^{g^{2}})+S_{8}(\beta^{g^{6}})=0$. So
\[
S_{8}(\beta^{g^{2}})=S_{8}(\beta^{g^{6}})=S_{8}(\beta^{g^{2}})S_{8}(\beta^{g^{6}})=(4,2)_{8}\in\{0,1\}.
\]

\end{itemize}
Let $c_{0}=\sum_{t=4}^{7}S_{8}(\beta^{g^{t}}).$ Always we have $c_{0}=y_{1}\mu$
or $y_{1}\mu+1$. By Corollary \ref{cor:Ad},
\[
\overrightarrow{A_{8}}(\beta)=(c_{0},c_{0},\dots,c_{0})+(1,1,1,1,0,0,0,0).
\]
Therefore, up to transformations $\sigma_{h}$, we have 
\[
\overrightarrow{A_{8}}(\beta)=\begin{cases}
(0,0,0,0,1,1,1,1), & \text{if }p\equiv1\ (\mathrm{mod}\ 16)\text{ and }y\equiv0\ (\mathrm{mod}\ 8),\\
(\mu,\mu,\mu,\mu,\mu+1,\mu+1,\mu+1,\mu+1), & \text{if }p\equiv1\ (\mathrm{mod}\ 16)\text{ and }y\equiv4\ (\mathrm{mod}\ 8).
\end{cases}
\]

\subsection{Case: $p\equiv9$ $(\mathrm{mod}\ 16)$ and $\left(\frac{2}{p}\right)_{4}=1$}
\begin{itemize}
\item $p\equiv9$ $(\mathrm{mod}\ 16)$ $\Leftrightarrow$ $x\equiv5$ $(\mathrm{mod}\ 8)$.
\item $\left(\frac{2}{p}\right)_{4}=1$ $\Leftrightarrow$ $y\equiv0$ $(\mathrm{mod}\ 4)$
$\Leftrightarrow$ $a\equiv1$ $(\mathrm{mod}\ 8)$.
\item $a^{2}\equiv1$ $(\mathrm{mod}\ 16)$ $\Rightarrow2b^{2}\equiv8$
$(\mathrm{mod}\ 16)$ $\Rightarrow b\equiv2$ $(\mathrm{mod}\ 4)$.
\end{itemize}
Let $x=8x_{1}+5$, $y=4y_{1}$, $a=8a_{1}+1$, and $b=4b_{1}+2$,
with integers $x_{1}$, $y_{1}$, $a_{1}$ and $b_{1}$. Note that
$p=(8x_{1}+5)^{2}+4(4y_{1})^{2}=(8a_{1}+1)^{2}+2(4b_{1}+2)^{2}$ $\Rightarrow$
$16x_{1}+25\equiv16a_{1}+32b_{1}(b_{1}+1)+9$ $(\mathrm{mod}\ 64)$
$\Rightarrow$ $a_{1}\equiv x_{1}+1$ $(\mathrm{mod}\ 4)$.

Note that $(4,1)_{8}-(4,2)_{8}=4x+12a=(32x_{1}+96a_{1}+32)/64\equiv(32+96)a_{1}/64\equiv0$
$(\mathrm{mod}\ 2)$, and
\begin{align*}
(4,1)_{8} & =(4,3)_{8}\equiv(4,2)_{8}\ (\mathrm{mod}\ 2),\\
(4,1)_{8} & =\frac{1}{64}(p-7+2x+4a)=x_{1}^{2}+y_{1}^{2}+\frac{3}{2}x_{1}+\frac{1}{2}a_{1}+\frac{1}{2}\equiv x_{1}+y_{1}+1\ (\mathrm{mod}\ 2),\\
(4,0)_{8} & =\frac{1}{64}(p-15-2x)=x_{1}^{2}+y_{1}^{2}+x_{1}\equiv y_{1}\ (\mathrm{mod}\ 2).
\end{align*}

By a good choice of $\alpha$, we may assume
\[
\overrightarrow{S_{4}}(\beta)=(S_{4}(\beta),S_{4}(\beta^{g}),S_{4}(\beta^{g^{2}}),S_{4}(\beta^{g^{3}}))=(0,1,1,1).
\]

\begin{itemize}
\item Note that $S_{8}(\beta)+S_{8}(\beta^{g^{4}})=0$, and
\[
S_{8}(\beta)=S_{8}(\beta^{g^{4}})=S_{8}(\beta)S_{8}(\beta^{g^{4}})=(4,1)_{8}+(4,2)_{8}+(4,3)_{8}+1=x_{1}+y_{1}\in\{0,1\}.
\]

\item For $i\in\{1,2,3\}$, $S_{8}(\beta^{g^{i}})+S_{8}(\beta^{g^{i+4}})=1$,
and
\begin{align*}
S_{8}(\beta^{g^{i}})S_{8}(\beta^{g^{i+4}}) & =(4,0)_{8}+2(4,1)_{8}+1=y_{1}+1,\\
\{S_{8}(\beta^{g^{i}}),S_{8}(\beta^{g^{i+4}})\} & =\{(y_{1}+1)\mu,(y_{1}+1)\mu+1\}.
\end{align*}

\end{itemize}
Let $c_{0}=\sum_{t=4}^{7}S_{8}(\beta^{g^{t}})$. Always we have $c_{0}=(y_{1}+1)\mu$
or $(y_{1}+1)\mu+1$. By Corollary \ref{cor:Ad},
\[
\overrightarrow{A_{8}}(\beta)=(c_{0},c_{0},\dots,c_{0})+(0,1,0,1,1,0,1,0).
\]
Therefore, up to transformations $\sigma_{h}$, we have 
\[
\overrightarrow{A_{8}}(\beta)=\begin{cases}
(0,1,0,1,1,0,1,0), & \text{if }p\equiv9\ (\mathrm{mod}\ 16)\text{ and }y\equiv4\ (\mathrm{mod}\ 8),\\
(\mu,\mu+1,\mu,\mu+1,\mu+1,\mu,\mu+1,\mu), & \text{if }p\equiv9\ (\mathrm{mod}\ 16)\text{ and }y\equiv0\ (\mathrm{mod}\ 8).
\end{cases}
\]

\subsection{Case: $p\equiv1$ $(\mathrm{mod}\ 16)$ and $\left(\frac{2}{p}\right)_{4}\protect\neq1$}
\begin{itemize}
\item $p\equiv1$ $(\mathrm{mod}\ 16)$ $\Leftrightarrow$ $x\equiv1$ $(\mathrm{mod}\ 8)$.
\item $\left(\frac{2}{p}\right)_{4}\neq1$ $\Leftrightarrow$ $y\equiv2$
$(\mathrm{mod}\ 4)$ $\Leftrightarrow$ $a\equiv5$ $(\mathrm{mod}\ 8)$.
\item $a^{2}\equiv9$ $(\mathrm{mod}\ 16)$ $\Rightarrow2b^{2}\equiv8$
$(\mathrm{mod}\ 16)$ $\Rightarrow b\equiv2$ $(\mathrm{mod}\ 4)$.
\end{itemize}
Let $x=8x_{1}+1$, $y=4y_{1}+2$, $a=8a_{1}+5$ and $b=4b_{1}+2$,
with integers $x_{1}$, $y_{1}$, $a_{1}$ and $b_{1}$. Note that
$p=(8x_{1}+1)^{2}+4(4y_{1}+2)^{2}=(8a_{1}+5)^{2}+2(4b_{1}+2)^{2}$
$\Rightarrow x_{1}-a_{1}\equiv1$ $(\mathrm{mod}\ 4)$. Also note
that
\begin{align*}
(4,0)_{8} & =\frac{1}{64}(p-7-10x)=x_{1}^{2}-x_{1}+y_{1}^{2}+y_{1}\equiv0\ (\mathrm{mod}\ 2),\\
(4,1)_{8} & =\frac{1}{64}(p+1+2x-4a+16y)=x_{1}^{2}+y_{1}^{2}+2y_{1}+\frac{1}{2}(1+x_{1}-a_{1})\equiv x_{1}+y_{1}+1\ (\mathrm{mod}\ 2),\\
(4,2)_{8} & =\frac{1}{64}(p+1+6x+8a)=x_{1}^{2}+y_{1}^{2}+x_{1}+y_{1}+a_{1}+1\equiv x_{1}\ (\mathrm{mod}\ 2),\\
(4,3)_{8} & =(4,1)_{8}-\frac{1}{2}y\equiv x_{1}+y_{1}+1-(2y_{1}+1)\equiv x_{1}+y_{1}\ (\mathrm{mod}\ 2).
\end{align*}

By a good choice of $\alpha$, we may assume
\[
\overrightarrow{S_{4}}(\beta)=(S_{4}(\beta),S_{4}(\beta^{g}),S_{4}(\beta^{g^{2}}),S_{4}(\beta^{g^{3}}))=(\mu,1,\mu+1,1).
\]

\begin{itemize}
\item Note that $S_{8}(\beta)+S_{8}(\beta^{g^{4}})=\mu$, and
\begin{align*}
S_{8}(\beta)S_{8}(\beta^{g^{4}}) & =((4,0)_{8}+(4,2)_{8})\mu+(4,1)_{8}+(4,2)_{8}+(4,3)_{8}=x_{1}\mu+x_{1}+1,\\
\{S_{8}(\beta),S_{8}(\beta^{g^{4}})\} & =\begin{cases}
\{\eta^{3}+\eta,\eta^{3}+\eta^{2}\}, & \text{if }x_{1}\equiv0\ (\mathrm{mod}\ 2),\\
\{\eta^{3}+\eta+1,\eta^{3}+\eta^{2}+1\}, & \text{if }x_{1}\equiv1\ (\mathrm{mod}\ 2).
\end{cases}
\end{align*}

\item Note that $S_{8}(\beta^{g^{2}})+S_{8}(\beta^{g^{6}})=\mu+1$, and
\begin{align*}
S_{8}(\beta^{g^{2}})S_{8}(\beta^{g^{6}}) & =((4,2)_{8}+(4,0)_{8})\mu+(4,3)_{8}+(4,0)_{8}+(4,1)_{8}=x_{1}\mu+1,\\
\{S_{8}(\beta^{g^{2}}),S_{8}(\beta^{g^{6}})\} & =\begin{cases}
\{\eta^{3},\eta^{3}+\eta^{2}+\eta+1\}, & \text{if }x_{1}\equiv0\ (\mathrm{mod}\ 2),\\
\{\eta^{3}+1,\eta^{3}+\eta^{2}+\eta\}, & \text{if }x_{1}\equiv1\ (\mathrm{mod}\ 2),
\end{cases}
\end{align*}

\item Note that $S_{8}(\beta^{g})+S_{8}(\beta^{g^{5}})=1$, and
\begin{align*}
S_{8}(\beta^{g})S_{8}(\beta^{g^{5}}) & =((4,3)_{8}+(4,1)_{8})\mu+(4,0)_{8}+(4,1)_{8}+(4,2)_{8}=\mu+(y_{1}+1),\\
\{S_{8}(\beta^{g}),S_{8}(\beta^{g^{5}})\} & =\begin{cases}
\{\eta^{2},\eta^{2}+1\}, & \text{if }y_{1}\equiv0\ (\mathrm{mod}\ 2),\\
\{\eta,\eta+1\}, & \text{if }y_{1}\equiv1\ (\mathrm{mod}\ 2),
\end{cases}
\end{align*}

\item Note that $S_{8}(\beta^{g^{3}})+S_{8}(\beta^{g^{7}})=1$, and
\begin{align*}
S_{8}(\beta^{g^{3}})S_{8}(\beta^{g^{7}}) & =((4,1)_{8}+(4,3)_{8})\mu+(4,2)_{8}+(4,3)_{8}+(4,0)_{8}=\mu+y_{1},\\
\{S_{8}(\beta^{g^{3}}),S_{8}(\beta^{g^{7}})\} & =\begin{cases}
\{\eta,\eta+1\}, & \text{if }y_{1}\equiv0\ (\mathrm{mod}\ 2),\\
\{\eta^{2},\eta^{2}+1\}, & \text{if }y_{1}\equiv1\ (\mathrm{mod}\ 2),
\end{cases}
\end{align*}

\end{itemize}
Note that $S_{8}(\beta^{g^{4}})+S_{8}(\beta^{g^{6}})\in\{\eta,\eta+1,\eta^{2},\eta^{2}+1\}$,
and $S_{8}(\beta^{g^{5}})+S_{8}(\beta^{g^{7}})\in\{\mu,\mu+1\}$.
Let $c_{0}=\sum_{t=4}^{7}S_{8}(\beta^{g^{t}})$. Always we have $c_{0}\in\{\eta,\eta+1,\eta^{2},\eta^{2}+1\}$.
By Corollary \ref{cor:Ad},
\[
\overrightarrow{A_{8}}(\beta)=(c_{0},c_{0},\dots,c_{0})+(\mu,\mu+1,0,1,\mu+1,\mu,1,0).
\]
Therefore, when $p\equiv1$ $(\mathrm{mod}\ 16)$ and $\left(\frac{2}{p}\right)_{4}\neq1$,
up to transformations $\sigma_{h}$, we have 
\begin{align*}
\overrightarrow{A_{8}}(\beta)= & \ (\eta,\eta^{2},\eta^{2}+1,\eta,\eta+1,\eta^{2}+1,\eta^{2},\eta+1)\\
\text{or} & \ (\eta^{2},\eta,\eta+1,\eta^{2},\eta^{2}+1,\eta+1,\eta,\eta^{2}+1).
\end{align*}

\subsection{Case: $p\equiv9$ $(\mathrm{mod}\ 16)$ and $\left(\frac{2}{p}\right)_{4}\protect\neq1$}
\begin{itemize}
\item $p\equiv9$ $(\mathrm{mod}\ 16)$ $\Leftrightarrow$ $x\equiv5$ $(\mathrm{mod}\ 8)$.
\item $\left(\frac{2}{p}\right)_{4}\neq1$ $\Leftrightarrow$ $y\equiv2$
$(\mathrm{mod}\ 4)$ $\Leftrightarrow$ $a\equiv5$ $(\mathrm{mod}\ 8)$.
\item $a^{2}\equiv9$ $(\mathrm{mod}\ 16)$ $\Rightarrow2b^{2}\equiv0$
$(\mathrm{mod}\ 16)$ $\Rightarrow b\equiv0$ $(\mathrm{mod}\ 4)$. 
\end{itemize}
Let $x=8x_{1}+5$, $y=4y_{1}+2$, $a=8a_{1}+5$ and $b=4b_{1}$, with
integers $x_{1}$, $y_{1}$, $a_{1}$ and $b_{1}$. Note that $p=(8x_{1}+5)^{2}+4(4y_{1}+2)^{2}=(8a_{1}+5)^{2}+2(4b_{1})^{2}$
$\Rightarrow x_{1}-a_{1}+1\equiv2b_{2}^{2}$ $(\mathrm{mod}\ 4)$.
Also note that
\begin{align*}
(4,0)_{8} & =\frac{1}{64}(p-15-10x-8a)=x_{1}^{2}+y_{1}^{2}+y_{1}-a_{1}-1\equiv0\ (\mathrm{mod}\ 2),\\
(4,1)_{8} & =\frac{1}{64}(p-7+2x+4a+16y)=x_{1}^{2}+y_{1}^{2}+\frac{3}{2}(x_{1}+1)+2y_{1}+\frac{1}{2}a_{1}\equiv x_{1}+y_{1}+b_{1}\ (\mathrm{mod}\ 2),\\
(4,2)_{8} & =\frac{1}{64}(p-7+6x)=x_{1}^{2}+y_{1}^{2}+2x_{1}+y_{1}+1\equiv x_{1}+1\ (\mathrm{mod}\ 2),\\
(4,3)_{8} & =(4,1)_{8}-\frac{1}{2}(4y_{1}+2)\equiv x_{1}+y_{1}+b_{1}+1\ (\mathrm{mod}\ 2).
\end{align*}

By a good choice of $\alpha$, we may assume
\[
\overrightarrow{S_{4}}(\beta)=(S_{4}(\beta),S_{4}(\beta^{g}),S_{4}(\beta^{g^{2}}),S_{4}(\beta^{g^{3}}))=(\mu,0,\mu+1,0).
\]

\begin{itemize}
\item Note that $S_{8}(\beta)+S_{8}(\beta^{g^{4}})=\mu$, and
\begin{align*}
S_{8}(\beta)S_{8}(\beta^{g^{4}}) & =(4,0)_{8}\mu+(4,2)_{8}(\mu+1)+1=(x_{1}+1)\mu+x_{1},\\
\{S_{8}(\beta),S_{8}(\beta^{g^{4}})\} & =\begin{cases}
\{\eta^{3}+\eta+1,\eta^{3}+\eta^{2}+1\}, & \text{if }x_{1}\equiv0\ (\mathrm{mod}\ 2),\\
\{\eta^{3}+\eta,\eta^{3}+\eta^{2}\}, & \text{if }x_{1}\equiv1\ (\mathrm{mod}\ 2).
\end{cases}
\end{align*}

\item Note that $S_{8}(\beta^{g^{2}})+S_{8}(\beta^{g^{6}})=\mu+1$, and
\begin{align*}
S_{8}(\beta^{g^{2}})S_{8}(\beta^{g^{6}}) & =(4,2)_{8}\mu+(4,0)_{8}(\mu+1)+1=(x_{1}+1)\mu+1,\\
\{S_{8}(\beta^{g^{2}}),S_{8}(\beta^{g^{6}})\} & =\begin{cases}
\{\eta^{3}+1,\eta^{3}+\eta^{2}+\eta\}, & \text{if }x_{1}\equiv0\ (\mathrm{mod}\ 2),\\
\{\eta^{3},\eta^{3}+\eta^{2}+\eta+1\}, & \text{if }x_{1}\equiv1\ (\mathrm{mod}\ 2).
\end{cases}
\end{align*}

\item Note that $S_{8}(\beta^{g})+S_{8}(\beta^{g^{5}})=0$, so
\[
S_{8}(\beta^{g})=S_{8}(\beta^{g^{5}})=S_{8}(\beta^{g})S_{8}(\beta^{g^{5}})=(4,3)_{8}\mu+(4,1)_{8}(\mu+1)+1=\mu+x_{1}+y_{1}+b_{1}+1.
\]

\item Note that $S_{8}(\beta^{g^{3}})+S_{8}(\beta^{g^{7}})=0$, so
\[
S_{8}(\beta^{g^{3}})=S_{8}(\beta^{g^{7}})=S_{8}(\beta^{g^{3}})S_{8}(\beta^{g^{7}})=(4,1)_{8}\mu+(4,3)_{8}(\mu+1)+1=\mu+x_{1}+y_{1}+b_{1}.
\]

\end{itemize}
Note that $S_{8}(\beta^{g^{4}})+S_{8}(\beta^{g^{6}})\in\{\eta,\eta+1,\eta^{2},\eta^{2}+1\}$,
and $S_{8}(\beta^{g^{5}})+S_{8}(\beta^{g^{7}})=1$. Let $c_{0}=\sum_{t=4}^{7}S_{8}(\beta^{g^{t}})$.
Always we have $c_{0}\in\{\eta,\eta+1,\eta^{2},\eta^{2}+1\}$. By
Corollary \ref{cor:Ad},
\[
\overrightarrow{A_{8}}(\beta)=(c_{0},c_{0},\dots,c_{0})+(\mu,\mu,1,1,\mu+1,\mu+1,0,0).
\]
Therefore, when $p\equiv9$ $(\mathrm{mod}\ 16)$ and $\left(\frac{2}{p}\right)_{4}\neq1$,
up to transformations $\sigma_{h}$, we have
\begin{align*}
\overrightarrow{A_{8}}(\beta)= & \ (\eta,\eta,\eta^{2},\eta^{2},\eta+1,\eta+1,\eta^{2}+1,\eta^{2}+1)\\
\text{or} & \ (\eta^{2},\eta^{2},\eta,\eta,\eta^{2}+1,\eta^{2}+1,\eta+1,\eta+1).
\end{align*}
In view of the above four cases, we get the following criterion for
$\left(\frac{2}{p}\right)_{8}=1$.
\begin{prop}
Let $p=x^{2}+4y^{2}$ be an odd prime with $p\equiv1$ $(\mathrm{mod}\ 8)$
and integers $x$,$y$. Let $g$ be a primitive root modulo $p$,
$\beta$ be a primitive $p$-th roots of unity in an extension field
of $\mathrm{GF}(2)$, $S_{8}(\beta^{g^{i}})=\sum_{t=0}^{\frac{p-1}{8}-1}\beta^{g^{i+8t}}$
and $s_{i,8}=\sum_{v=4}^{7}S_{8}(\beta^{g^{i+v}})$ for $0\leqslant i\leqslant7$.
Then the following conditions are equivalent.

$(1)$$\left(\frac{2}{p}\right)_{8}=1$, i.e., $2$ is an octic residue
modulo $p$.

$(2)$ $S_{8}(\beta^{g^{i}})\in\{0,1\}$ for any $0\leqslant i\leqslant7$.

$(3)$ $s_{i,8}\in\{0,1\}$ for any $0\leqslant i\leqslant7$.

$(4)$ Either $p\equiv1$ $(\mathrm{mod}\ 16)$ and $y\equiv0$ $(\mathrm{mod}\ 8)$,
or $p\equiv9$ $(\mathrm{mod}\ 16)$ and $y\equiv4$ $(\mathrm{mod}\ 8)$.\end{prop}
\begin{proof}
$(1)\Rightarrow(2)$: As $\left(\frac{2}{p}\right)_{8}=1$, let $2\equiv g^{8u}$
$(\mathrm{mod}\ p)$ with $u\in\mathbb{Z}$. For $0\leqslant i\leqslant7$.
$S_{8}(\beta^{g^{i}})^{2}=S_{8}(\beta^{2g^{i}})=S_{8}(\beta^{g^{i+8u}})=S_{8}(\beta^{g^{i}})$,
so $S_{8}(\beta^{i})\in\{0,1\}$.

$(2)\Rightarrow(3)$ is obvious. $(3)\Leftrightarrow(4)$ by the above
expressions of $\overrightarrow{A_{8}}(\beta)=(s_{0,8},s_{1,8},\dots,s_{7,8})$
in four cases.

$(4)\Rightarrow(1)$: Let $2\equiv g^{l}$ $(\mathrm{mod}\ p)$ with
$l\in\mathbb{Z}$. Note that
\[
(s_{i,8}^{2})_{i=0}^{7}=\Bigl(\sum_{v=4}^{7}S_{8}(\beta^{2g^{i+v}})\Bigr)_{i=0}^{7}=\Bigl(\sum_{v=4}^{7}S_{8}(\beta^{g^{l+i+v}})\Bigr)_{i=0}^{7}=\sigma_{l}\Bigl(\sum_{v=4}^{7}S_{8}(\beta^{g^{i+v}})\Bigr)_{i=0}^{7}=\sigma_{l}(s_{i,8})_{i=0}^{7}.
\]
When $(4)$ holds, we may assume $(s_{i,8})_{i=0}^{7}=(0,0,0,0,1,1,1,1)$
or $(0,1,0,1,1,0,1,0)$, which should be invariant under the action
of $\sigma_{l}$. So $l\equiv0$ $(\mathrm{mod}\ 8)$, and thus $\left(\frac{2}{p}\right)_{8}=1$.
\end{proof}

\section{\label{sec:LCDH8}Linear complexity of DH-GCS$_{8}$}

This section is a proof of Theorem \ref{thm:main}. Let $s^{\infty}$
be the Ding-Helleseth generalized cyclotomic sequence DH-GCS$_{8}$,
with the matrix $\mathbb{S}=(s_{i,j})_{0\leqslant i,j\leqslant8}$.
Note that $s_{8,8}=0$ as $p\equiv q\equiv1$ $(\mathrm{mod}\ 8)$,
by Lemma \ref{lem:method}, 
\begin{align*}
L(s^{\infty}) & =pq-1-\frac{(p-1)(q-1)}{64}\cdot|\{0\leqslant i,j\leqslant7:s_{i,j}=0\}|\\
 & \quad-\frac{p-1}{8}\cdot|\{0\leqslant i\leqslant7:s_{i,8}=0\}|-\frac{q-1}{8}\cdot|\{0\leqslant j\leqslant7:s_{8,j}=0\}|,
\end{align*}
where $s_{i,j}=s_{i,8}+s_{8,j}+s_{8,j'}$ for $j'\in\{0,1,\dots,7\}$
with $j'\equiv j-\mathrm{ind}_{g}^{(q)}p$ $(\mathrm{mod}\ 8)$.

Section \ref{sec:A8} gives explicit expressions of $\overrightarrow{A_{8}}(\beta)=(s_{i,8})_{i=0}^{7}$.
Clearly, if we replace $p$ by $q$, and replace $\beta$ by $\gamma$,
then we get $\overrightarrow{B_{8}}(\gamma)=(s_{8,j})_{j=0}^{7}$
as in Table \ref{tab:B8}.
\begin{table}[h]
\protect\caption{\label{tab:B8} $\protect\overrightarrow{B_{8}}(\gamma)=(s_{8,j})_{j=0}^{7}$
up to transformations $\sigma_{h}$}

\centering{}%
\begin{tabular}{c|c|c}
\hline 
 & $q\ \mathrm{mod}\ 16$ & $\overrightarrow{B_{8}}(\gamma)$\tabularnewline
\hline 
\hline 
\multirow{2}{*}{$\left(\frac{2}{q}\right)_{8}=1$} & $1$ & $(0,0,0,0,1,1,1,1)$\tabularnewline
\cline{2-3} 
 & $9$ & $(0,1,0,1,1,0,1,0)$\tabularnewline
\hline 
\multirow{2}{*}{$\left(\frac{2}{q}\right)_{4}=1\neq\left(\frac{2}{q}\right)_{8}$} & $1$ & $(\mu,\mu,\mu,\mu,\mu+1,\mu+1,\mu+1,\mu+1)$\tabularnewline
\cline{2-3} 
 & $9$ & $(\mu,\mu+1,\mu,\mu+1,\mu+1,\mu,\mu+1,\mu)$\tabularnewline
\hline 
\multirow{4}{*}{$\left(\frac{2}{q}\right)_{4}\neq1$} & \multirow{2}{*}{$1$} & $(\eta,\eta^{2},\eta^{2}+1,\eta,\eta+1,\eta^{2}+1,\eta^{2},\eta+1)$
or\tabularnewline
 &  & $(\eta^{2},\eta,\eta+1,\eta^{2},\eta^{2}+1,\eta+1,\eta,\eta^{2}+1)$\tabularnewline
\cline{2-3} 
 & \multirow{2}{*}{$9$} & $(\eta,\eta,\eta^{2},\eta^{2},\eta+1,\eta+1,\eta^{2}+1,\eta^{2}+1)$
or\tabularnewline
 &  & $(\eta^{2},\eta^{2},\eta,\eta,\eta^{2}+1,\eta^{2}+1,\eta+1,\eta+1)$\tabularnewline
\hline 
\end{tabular}
\end{table}

Therefore, $\overrightarrow{B_{8}}(\gamma)+\sigma_{-\mathrm{ind}_{g}^{(q)}p}(\overrightarrow{B_{8}}(\gamma))=(s_{8,j}+s_{8,j'})_{j=0}^{7}$
are as in Table \ref{tab:BB}, with
\[
\mathrm{ind}_{g}^{(q)}p\equiv\begin{cases}
0 & (\mathrm{mod}\ 8),\ \text{if }\left(\frac{p}{q}\right)_{8}=1,\\
4 & (\mathrm{mod}\ 8),\ \text{if }\left(\frac{p}{q}\right)_{8}\neq1\text{ and }\left(\frac{p}{q}\right)_{4}=1,\\
\pm2 & (\mathrm{mod}\ 8),\ \text{if }\left(\frac{p}{q}\right)_{4}\neq1\text{ and }\left(\frac{p}{q}\right)_{2}=1,\\
\pm1\text{ or}\pm3 & (\mathrm{mod}\ 8),\ \text{if }\left(\frac{p}{q}\right)_{2}\neq1.
\end{cases}
\]
\begin{table}[h]
\protect\caption{\label{tab:BB}$\protect\overrightarrow{B_{8}}(\gamma)+\sigma_{-\mathrm{ind}_{g}^{(q)}p}(\protect\overrightarrow{B_{8}}(\gamma))=(s_{8,j}+s_{8,j'})_{j=0}^{7}$
up to transformations $\sigma_{h}$}

\centering{}%
\begin{tabular}{c|c|c}
\hline 
$ $ & \multicolumn{1}{c|}{$\left(\frac{2}{q}\right)_{4}=1$} & \multicolumn{1}{c}{$\left(\frac{2}{q}\right)_{4}\neq1$}\tabularnewline
\hline 
\hline 
$\mathrm{ind}_{g}^{(q)}p\equiv0$ $\mathrm{mod}\ 8$ & \multicolumn{2}{c}{$(0,0,0,0,0,0,0,0)$}\tabularnewline
\hline 
$\mathrm{ind}_{g}^{(q)}p\equiv4$ $\mathrm{mod}\ 8$ & \multicolumn{2}{c}{$(1,1,1,1,1,1,1,1)$}\tabularnewline
\hline 
$\mathrm{ind}_{g}^{(q)}p\equiv\pm2$ $\mathrm{mod}\ 8$ & \multicolumn{1}{c|}{$(1,1,0,0,1,1,0,0)$} & \multicolumn{1}{c}{$(\mu,\mu,\mu+1,\mu+1,\mu,\mu,\mu+1,\mu+1)$}\tabularnewline
\hline 
$\mathrm{ind}_{g}^{(q)}p\equiv\pm1$ $\mathrm{mod}\ 8$, $q\equiv1\ (\mathrm{mod}\ 16)$ & \multirow{2}{*}{$(1,0,0,0,1,0,0,0)$} & \multirow{2}{*}{$(\mu,1,\mu+1,1,\mu,1,\mu+1,1)$}\tabularnewline
\cline{1-1} 
$\mathrm{ind}_{g}^{(q)}p\equiv\pm3$ $\mathrm{mod}\ 8$, $q\equiv9\ (\mathrm{mod}\ 16)$ &  & \tabularnewline
\hline 
$\mathrm{ind}_{g}^{(q)}p\equiv\pm1$ $\mathrm{mod}\ 8$, $q\equiv9\ (\mathrm{mod}\ 16)$ & \multirow{2}{*}{$(1,1,1,0,1,1,1,0)$} & \multirow{2}{*}{$(\mu,0,\mu+1,0,\mu,0,\mu+1,0)$}\tabularnewline
\cline{1-1} 
$\mathrm{ind}_{g}^{(q)}p\equiv\pm3$ $\mathrm{mod}\ 8$, $q\equiv1\ (\mathrm{mod}\ 16)$ &  & \tabularnewline
\hline 
\end{tabular}
\end{table}

Note that $|\{0\leqslant i\leqslant7:s_{i,8}=0\}|=\begin{cases}
4, & \text{if }\left(\frac{2}{p}\right)_{8}=1,\\
0, & \text{otherwise},
\end{cases}$ and $|\{0\leqslant j\leqslant7:s_{8,j}=0\}|=\begin{cases}
4, & \text{if }\left(\frac{2}{q}\right)_{8}=1,\\
0, & \text{otherwise}.
\end{cases}$
\begin{itemize}
\item For $0\leqslant i\leqslant7$,
\[
s_{i,8}=\begin{cases}
0\text{ or }1\text{ (with each occurring }4\text{ times)}, & \text{if }\left(\frac{2}{p}\right)_{8}=1,\\
\mu\text{ or }\mu+1\text{ (with each occurring }4\text{ times)}, & \text{if }\left(\frac{2}{p}\right)_{8}\neq1\text{ and }\left(\frac{2}{p}\right)_{4}=1,\\
\eta,\eta^{2},\eta+1\text{ or }\eta^{2}+1\text{ (with each occurring }4\text{ times)}, & \text{if }\left(\frac{2}{p}\right)_{4}\neq1.
\end{cases}
\]

\item For $0\leqslant j,j'\leqslant7$ with $j'\equiv j-\mathrm{ind}_{g}^{(q)}p$
$(\mathrm{mod}\ 8)$,
\[
s_{8,j}+s_{8,j'}=\begin{cases}
0\text{ or }1, & \text{if }\left(\frac{2}{q}\right)_{4}=1\text{ or }\left(\frac{p}{q}\right)_{4}=1,\\
\mu\text{ or }\mu+1, & \text{if }\left(\frac{2}{q}\right)_{4}\neq1,\left(\frac{p}{q}\right)_{4}\neq1\text{ and }\left(\frac{p}{q}\right)_{2}=1,\\
0,1,\mu\text{ or }\mu+1, & \text{if }\left(\frac{2}{q}\right)_{4}\neq1\text{ and }\left(\frac{p}{q}\right)_{2}\neq1,
\end{cases}
\]
Especially when $\left(\frac{2}{q}\right)_{4}\neq1$ and $\left(\frac{p}{q}\right)_{2}\neq1$,
\begin{align*}
2 & =|\{0\leqslant j\leqslant7:s_{8,j}+s_{8,j'}=0\text{ or }1\}|\\
 & =|\{0\leqslant j\leqslant7:s_{8,j}+s_{8,j'}=\mu\text{ or }\mu+1\}|.
\end{align*}

\end{itemize}
Therefore, we count pairs $(i,j)$ in $\{0,1,\dots,7\}^{2}$ satisfying
$0=s_{i,j}=s_{i,8}+s_{8,j}+s_{8,j'}$ as in Table \ref{tab:sij=00003D0}.
\begin{table}[h]
\centering{}\protect\caption{\label{tab:sij=00003D0}The value of $|\{0\leqslant i,j\leqslant7:s_{i,j}=0\}|/64$}
\begin{tabular}{c|c|c|c}
\hline 
\multirow{1}{*}{} & $\left(\frac{2}{p}\right)_{8}=1$ & $\left(\frac{2}{p}\right)_{4}=1\neq\left(\frac{2}{p}\right)_{8}$ & $\left(\frac{2}{p}\right)_{4}\neq1$\tabularnewline
\hline 
\hline 
$\left(\frac{2}{q}\right)_{4}=1$ or $\left(\frac{p}{q}\right)_{4}=1$ & $\frac{1}{2}$ & $0$ & \multirow{3}{*}{$0$}\tabularnewline
\cline{1-3} 
$\left(\frac{2}{q}\right)_{4}\neq1$ and $\left(\frac{p}{q}\right)_{2}=1\neq\left(\frac{p}{q}\right)_{4}$ & $0$ & $\frac{1}{2}$ & \tabularnewline
\cline{1-3} 
$\left(\frac{2}{q}\right)_{4}\neq1$ and $\left(\frac{p}{q}\right)_{2}\neq1$ & \multicolumn{2}{c|}{$\frac{1}{4}$} & \tabularnewline
\hline 
\end{tabular}
\end{table}

In fact, the above argument gives an accounting of the zeros in the
matrix $\mathbb{S}=(s_{i,j})_{i,j=0}^{8}$, and can be easily translated
into the statement of Theorem \ref{thm:main}. Moreover, the result
can be rewritten more precisely as the following Theorem \ref{thm:cases}.
\begin{thm}
\label{thm:cases} The linear complexity of the Ding-Helleseth generalized
cyclotomic sequence of length $pq$ and order $\mathrm{gcd}(p-1,q-1)=8$
for primes $p$ and $q$ is
\[
L(p,q)=\begin{cases}
pq-1, & \text{if }\left(\frac{2}{p}\right)_{4}\neq1,\left(\frac{2}{q}\right)_{8}\neq1;\\
pq-1, & \text{if }\left(\frac{2}{p}\right)_{4}=1\neq\left(\frac{2}{p}\right)_{8},\left(\frac{2}{q}\right)_{4}=1\neq\left(\frac{2}{q}\right)_{8};\\
pq-1, & \text{if }\left(\frac{2}{p}\right)_{4}=1\neq\left(\frac{2}{p}\right)_{8},\left(\frac{2}{q}\right)_{4}\neq1,\left(\frac{p}{q}\right)_{4}=1;\\
pq-1-\frac{p-1}{2}, & \text{if }\left(\frac{2}{p}\right)_{8}=1,\left(\frac{2}{q}\right)_{4}\neq1,\left(\frac{p}{q}\right)_{2}=1\neq\left(\frac{p}{q}\right)_{4};\\
pq-1-\frac{q-1}{2}, & \text{if }\left(\frac{2}{p}\right)_{4}\neq1,\left(\frac{2}{q}\right)_{8}=1;\\
pq-1-\frac{q-1}{2}, & \text{if }\left(\frac{2}{p}\right)_{4}=1\neq\left(\frac{2}{p}\right)_{8},\left(\frac{2}{q}\right)_{8}=1;\\
pq-1-\frac{(p-1)(q-1)}{2}, & \text{if }\left(\frac{2}{p}\right)_{4}=1\neq\left(\frac{2}{p}\right)_{8},\left(\frac{2}{q}\right)_{4}\neq1,\left(\frac{p}{q}\right)_{2}=1\neq\left(\frac{p}{q}\right)_{4};\\
pq-1-\frac{(p-1)(q-1)}{2}-\frac{p-1}{2}, & \text{if }\left(\frac{2}{p}\right)_{8}=1,\left(\frac{2}{q}\right)_{4}=1\neq\left(\frac{2}{q}\right)_{8};\\
pq-1-\frac{(p-1)(q-1)}{2}-\frac{p-1}{2}, & \text{if }\left(\frac{2}{p}\right)_{8}=1,\left(\frac{2}{q}\right)_{4}\neq1,\left(\frac{p}{q}\right)_{4}=1;\\
pq-1-\frac{(p-1)(q-1)}{2}-\frac{p-1}{2}-\frac{q-1}{2}, & \text{if }\left(\frac{2}{p}\right)_{8}=1,\left(\frac{2}{q}\right)_{8}=1;\\
pq-1-\frac{(p-1)(q-1)}{4}, & \text{if }\left(\frac{2}{p}\right)_{4}=1\neq\left(\frac{2}{p}\right)_{8},\left(\frac{2}{q}\right)_{4}\neq1,\left(\frac{p}{q}\right)_{2}\neq1;\\
pq-1-\frac{(p-1)(q-1)}{4}-\frac{p-1}{2}, & \text{if }\left(\frac{2}{p}\right)_{8}=1,\left(\frac{2}{q}\right)_{4}\neq1,\left(\frac{p}{q}\right)_{2}\neq1.
\end{cases}
\]
\end{thm}
\begin{rem*}
The twelve cases in Theorem \ref{thm:cases} all happen for some primes
$p$ and $q$ with $\mathrm{gcd}(p-1,q-1)=8$. Indeed, if we define
a function
\[
\mathrm{Res}(c,p)=\begin{cases}
0, & \text{if }c\equiv0\ (\mathrm{mod}\ p),\\
1, & \text{if }\left(\frac{c}{p}\right)_{2}\neq1,\\
2, & \text{if }\left(\frac{c}{p}\right)_{2}=1\neq\left(\frac{c}{p}\right)_{4},\\
4, & \text{if }\left(\frac{c}{p}\right)_{4}=1\neq\left(\frac{c}{p}\right)_{8},\\
8, & \text{if }\left(\frac{c}{p}\right)_{8}=1,
\end{cases}
\]
for $c\in\mathbb{Z}$ and a prime $p\equiv1$ $(\mathrm{mod}\ 8)$,
then the three values of $\mathrm{Res}(2,p)\in\{2,4,8\}$, $\mathrm{Res}(2,q)\in\{2,4,8\}$
and $\mathrm{Res}(p,q)\in\{1,2,4,8\}$ are independent, with examples
in Table \ref{tab:Res}. Note here that $\left(\frac{2}{p}\right)_{2}=\left(\frac{2}{q}\right)_{2}=1$
as $p\equiv q\equiv1$ $(\mathrm{mod}\ 8)$.
\begin{table}[h]
\centering{}\protect\caption{\label{tab:Res}Examples of $(p,q)$ for all $(\mathrm{Res}(2,p),\mathrm{Res}(2,q),\mathrm{Res}(p,q))$}
\begin{tabular}{c|c|c|c|c|c}
\hline 
$\mathrm{Res}(2,p)$ & $\mathrm{Res}(2,q)$ & $\mathrm{Res}(p,q)=1$ & $\mathrm{Res}(p,q)=2$ & $\mathrm{Res}(p,q)=4$ & $\mathrm{Res}(p,q)=8$\tabularnewline
\hline 
\hline 
\multirow{3}{*}{$2$} & $2$ & $(17,41)$ & $(17,137)$ & $(17,457)$ & $(17,409)$\tabularnewline
\cline{2-6} 
 & $4$ & $(17,617)$ & $(17,281)$ & $(17,1481)$ & $(41,2273)$\tabularnewline
\cline{2-6} 
 & $8$ & $(17,73)$ & $(17,89)$ & $(41,73)$ & $(17,1721)$\tabularnewline
\hline 
\multirow{3}{*}{$4$} & $2$ & $(113,137)$ & $(113,521)$ & $(113,41)$ & $(113,313)$\tabularnewline
\cline{2-6} 
 & $4$ & $(113,1097)$ & $(113,1049)$ & $(113,2473)$ & $(113,1033)$\tabularnewline
\cline{2-6} 
 & $8$ & $(113,73)$ & $(113,233)$ & $(113,1801)$ & $(113,1721)$\tabularnewline
\hline 
\multirow{3}{*}{$8$} & $2$ & $(73,17)$ & $(73,41)$ & $(73,809)$ & $(73,137)$\tabularnewline
\cline{2-6} 
 & $4$ & $(73,113)$ & $(73,593)$ & $(73,353)$ & $(73,1889)$\tabularnewline
\cline{2-6} 
 & $8$ & $(73,233)$ & $(73,1217)$ & $(73,89)$ & $(73,2969)$\tabularnewline
\hline 
\end{tabular}
\end{table}
 
\end{rem*}

\section{Verifying examples by SageMath}

This section illustrates the validity of Theorem \ref{thm:main} by
testing examples, with the help of SageMath \cite{sagemath}, a free
open-source mathematics software system based on Python and many open-source
packages. Let us write down SageMath codes to compute two things:
\begin{itemize}
\item The linear complexity as $LC(p,q,g)=pq-\mathrm{deg}(\mathrm{gcd}(x^{pq}-1,\sum_{i=0}^{pq-1}s_{i}x^{i}))$,
for the Ding-Helleseth generalized cyclotomic sequence DH-GCS$_{d}=(s_{i})_{i=0}^{\infty}$
with respect to odd primes $p,q$ and a common primitive root $g$
of theirs, where $d=\mathrm{gcd}(p-1,q-1)$.
\end{itemize}
\begin{lstlisting}
def LC(p,q,g):
    f = mod(g,p).crt(mod(1,q)); K.<x> = GF(2)[]
    d = gcd(p-1,q-1); e = (p-1)*(q-1)/d; S = 0 
    for i in range(d/2,d):
        for t in range(e/d):
            gidt = g^(i+d*t)
            for v in range(d): S = S + x^((gidt * f^v)%(p*q))
        for t in range((p-1)/d): S = S + x^(q * (g^(i+d*t)%p))
        for t in range((q-1)/d): S = S + x^(p * (g^(i+d*t)%q))
    return p*q - gcd(x^(p*q)-1,S).degree()
\end{lstlisting}

\begin{itemize}
\item The result $L(p,q)$ in Theorem \ref{thm:main} for odd primes $p$
and $q$ with $8=d=\mathrm{gcd}(p-1,q-1)$.
\end{itemize}
\begin{lstlisting}
def L(p,q):
    rp = mod(2,p).log(); rq = mod(2,q).log(); rpq = mod(p,q).log();
    L = p*q - 1
    if rq%8==0: L = L - (q-1)/2
    if rp%8==0:
        L = L - (p-1)/2
        if rq%4==0 or rpq%4==0: L = L - (p-1)*(q-1)/2
    if rp%8==4 and rq%4!=0 and rpq%4==2: L = L - (p-1)*(q-1)/2
    if rp%4==0 and rq%4!=0 and rpq%2==1: L = L - (p-1)*(q-1)/4
    return L
\end{lstlisting}

The following SageMath codes verify Theorem \ref{thm:main} for odd
primes $p,q\leqslant D$ with a common primitive root $g$ given randomly.

\begin{lstlisting}
D = 500; import random
for p in prime_range(D):
    for q in prime_range(p,D):
        if gcd(p-1,q-1) != 8: continue
        Lpq = L(p,q); Lqp = L(q,p)
        R = random.randint(0,p*q)
        for g in range(R,R+p*q):
            mp = mod(g,p).is_primitive_root()
            mq = mod(g,q).is_primitive_root()
            if mp == mq == True: break
        if LC(p,q,g) != Lpq:
            print ("Wrong for p = %s, q = %s, g = %s" % (p,q,g))
        if LC(q,p,g) != Lqp:
            print ("Wrong for p = %s, q = %s, g = %s" % (q,p,g))
        else:
            print ("p=%3s, q=%3s, g=%6s, L(p,q)=%6s, L(q,p)=%6s" %
                   (p,q,g,Lpq,Lqp))
\end{lstlisting}

Without installation of SageMath, one can simply run the above codes
online in \href{http://sagecell.sagemath.org}{SageMathCell} or \href{https://cocalc.com}{CoCalc}.
There is a shared web page of the codes and the output of a run on
CoCalc:

\href{https://cocalc.com/share/6ac1f09c-5057-4be6-bd80-8bd567792ee9/Pub.sagews?viewer=share}{https://cocalc.com/share/6ac1f09c-5057-4be6-bd80-8bd567792ee9/Pub.sagews?viewer=share}

\section{Conclusion}

This paper precisely determines the linear complexity of the Ding-Helleseth
generalized cyclotomic sequence DH-GCS$_{8}$ of length $pq$, explicit
in terms of $p$ and $q$, as in Theorem \ref{thm:main} or Theorem
\ref{thm:cases}. Especially we obtain the following facts.
\begin{itemize}
\item The linear complexity is independent of the choice of a common primitive
root $g$ of $p$ and $q$.
\item The linear complexity has an expression
\[
L(p,q)=pq-1-\varepsilon(p-1)-\kappa(q-1)-\eta(p-1)(q-1)
\]
with $\varepsilon,\kappa\in\{\frac{1}{2},0\}$ and $\eta\in\{\frac{1}{2},\frac{1}{4},0\}$
depending only on the values of $\mathrm{Res}(2,p),\mathrm{Res}(2,q)\in\{2,4,8\}$
and $\mathrm{Res}(p,q)\in\{1,2,4,8\}$.
\item The linear complexity is no less than $pq-1-\frac{1}{2}(p-1+q-1+(p-1)(q-1))=\frac{pq-1}{2}$,
which confirms the low bound in Yan \cite{Yan2011linear}. This is
a high linear complexity to resist security attacks of the Berlekamp--Massey
algorithm.
\end{itemize}

We can also read off the minimal polynomial $m(x)$ of DH-GCS$_{8}$
from our calculation. For given primes $p$ and $q$ with $\mathrm{gcd}(p-1,q-1)=8$,
by a good choice of a primitive $pq$-th root $\alpha$ of unity in
an extension field of $\mathrm{GF}(2)$, the matrix $\mathbb{S}=(s_{i,j})_{i,j=0}^{8}$
can be precisely determined by the expressions of $\overrightarrow{A_{8}}(\beta)=(s_{i,8})_{i=0}^{7}$
in Section \ref{sec:A8}, $\overrightarrow{B_{8}}(\gamma)=(s_{8,j})_{j=0}^{7}$
in Table \ref{tab:B8}, and $\overrightarrow{B_{8}}(\gamma)+\sigma_{-\mathrm{ind}_{g}^{(q)}p}(\overrightarrow{B_{8}}(\gamma))=(s_{8,j}+s_{8,j'})_{j=0}^{7}$
in Table \ref{tab:BB}. Then 
\[
m(x)=(x^{pq}-1)/((x-1)\prod_{\substack{0\leqslant i,j\leqslant7,\\
s_{i,j}=0
}
}\mathbf{D}_{i,j}(x)\prod_{\substack{0\leqslant j\leqslant7,\\
s_{8,j}=0
}
}\mathbf{P}_{j-\mathrm{ind}_{g}^{(q)}p}(x)\prod_{\substack{0\leqslant i\leqslant7,\\
s_{i,8}=0
}
}\mathbf{Q}_{i-\mathrm{ind}_{g}^{(p)}q}(x)),
\]
where $\mathbf{D}_{i,j}(x)=\prod_{k\in D_{i,j}}(x-\alpha^{k})$, $\mathbf{P}_{j}(x)=\prod_{k\in P_{j}}(x-\alpha^{k})$,
and $\mathbf{Q}_{i}(x)=\prod_{k\in Q_{i}}(x-\alpha^{k})$.

In the end, let us list the linear complexity $L(p,q)$ for all primes
$p,q\leqslant D=500$ with $\mathrm{gcd}(p-1,q-1)=8$, given by SageMath.
This list can extend to much larger $D$, but we choose $D=500$ just
because of the limitation of space to write here.
\begin{table}[H]
\protect\caption{$L(p,q)$ for primes $p,q\leqslant D=500$ with $\mathrm{gcd}(p-1,q-1)=8$}

\begin{center}
\begin{tabular}{cccc|cccc|cccc}
\hline 
$p$ & $q$ & $L(p,q)$ & $L(q,p)$ & $p$ & $q$ & $L(p,q)$ & $L(q,p)$ & $p$ & $q$ & $L(p,q)$ & $L(q,p)$\tabularnewline
\hline 
\hline 

17 & 41 & 696 & 696 & 
17 & 73 & 1204 & 916 & 
17 & 89 & 1468 & 764 \tabularnewline
17 & 137 & 2328 & 2328 & 
17 & 233 & 3844 & 2916 & 
17 & 281 & 4776 & 2536 \tabularnewline
17 & 313 & 5320 & 5320 & 
17 & 409 & 6952 & 6952 & 
17 & 457 & 7768 & 7768 \tabularnewline
41 & 73 & 2956 & 2956 & 
41 & 89 & 3604 & 2724 & 
41 & 97 & 3976 & 3976 \tabularnewline
41 & 113 & 4632 & 4632 & 
41 & 137 & 5616 & 5616 & 
41 & 193 & 7912 & 7912 \tabularnewline
41 & 233 & 9436 & 7116 & 
41 & 257 & 10408 & 7848 & 
41 & 313 & 12832 & 12832 \tabularnewline
41 & 337 & 13648 & 13648 & 
41 & 353 & 14472 & 14472 & 
41 & 409 & 16768 & 16768 \tabularnewline
41 & 433 & 17752 & 17752 & 
41 & 449 & 18408 & 18408 & 
41 & 457 & 18736 & 18736 \tabularnewline
73 & 89 & 3248 & 3248 & 
73 & 113 & 4180 & 8212 & 
73 & 137 & 5068 & 9964 \tabularnewline
73 & 233 & 8504 & 8504 & 
73 & 257 & 9380 & 9380 & 
73 & 281 & 10396 & 20476 \tabularnewline
73 & 353 & 13060 & 25732 & 
73 & 401 & 29236 & 29236 & 
73 & 449 & 24676 & 32740 \tabularnewline
89 & 97 & 8588 & 8588 & 
89 & 113 & 5084 & 10012 & 
89 & 137 & 9156 & 12148 \tabularnewline
89 & 193 & 12908 & 17132 & 
89 & 233 & 10368 & 10368 & 
89 & 241 & 16124 & 21404 \tabularnewline
89 & 257 & 11436 & 11436 & 
89 & 281 & 12644 & 24964 & 
89 & 313 & 20948 & 27812 \tabularnewline
89 & 337 & 14996 & 14996 & 
89 & 401 & 35644 & 35644 & 
89 & 409 & 18404 & 36356 \tabularnewline
89 & 433 & 28988 & 38492 & 
89 & 449 & 20204 & 39916 & 
89 & 457 & 30596 & 40628 \tabularnewline
97 & 137 & 13288 & 13288 & 
97 & 233 & 22484 & 16916 & 
97 & 281 & 27256 & 20536 \tabularnewline
113 & 137 & 11672 & 15480 & 
113 & 233 & 26212 & 13220 & 
113 & 313 & 35368 & 35368 \tabularnewline
113 & 409 & 34792 & 46216 & 
113 & 457 & 38872 & 51640 & 
137 & 193 & 26440 & 26440 \tabularnewline
137 & 233 & 31804 & 23916 & 
137 & 241 & 33016 & 33016 & 
137 & 257 & 35080 & 35080 \tabularnewline
137 & 281 & 38496 & 19456 & 
137 & 313 & 42880 & 42880 & 
137 & 337 & 46000 & 23152 \tabularnewline
137 & 353 & 48360 & 36392 & 
137 & 401 & 54936 & 54936 & 
137 & 433 & 59320 & 59320 \tabularnewline
137 & 449 & 61512 & 61512 & 
137 & 457 & 62608 & 62608 & 
193 & 233 & 44852 & 33716 \tabularnewline
193 & 281 & 54232 & 40792 & 
233 & 241 & 56036 & 56036 & 
233 & 257 & 29940 & 29940 \tabularnewline
233 & 281 & 32876 & 65356 & 
233 & 313 & 54716 & 72812 & 
233 & 337 & 39260 & 39260 \tabularnewline
233 & 353 & 41300 & 82132 & 
233 & 401 & 70116 & 93316 & 
233 & 409 & 71516 & 95180 \tabularnewline
233 & 433 & 50660 & 100772 & 
233 & 449 & 78516 & 104500 & 
233 & 457 & 106364 & 106364 \tabularnewline
257 & 281 & 36248 & 72088 & 
257 & 313 & 60344 & 80312 & 
257 & 409 & 78872 & 104984 \tabularnewline
257 & 457 & 58952 & 117320 & 
281 & 313 & 44272 & 87952 & 
281 & 353 & 99192 & 99192 \tabularnewline
281 & 409 & 57808 & 114928 & 
281 & 433 & 91432 & 121672 & 
281 & 457 & 96496 & 128416 \tabularnewline
313 & 353 & 110488 & 83032 & 
313 & 401 & 125512 & 125512 & 
313 & 449 & 140536 & 140536 \tabularnewline
353 & 409 & 108472 & 144376 & 
353 & 457 & 121192 & 161320 & 
401 & 409 & 164008 & 164008 \tabularnewline
401 & 457 & 183256 & 183256 & 
409 & 449 & 183640 & 183640 & 
449 & 457 & 205192 & 205192 \tabularnewline

\hline 
\end{tabular}
\end{center}
\end{table}

\bibliographystyle{model5-names}
\phantomsection\addcontentsline{toc}{section}{\refname}\bibliography{LC}

\end{document}